\documentclass{article}
\usepackage{amssymb,amsmath,amsthm,ascmac}
\usepackage{enumerate}

\newtheorem{thm}{Theorem}

\newtheorem{lem}{Lemma}[section]
\newtheorem{pro}{Proposition}[section]
\newtheorem{rem}{Remark}

\newcommand{\dis}{\displaystyle}

\newcommand{\R}{{\Bbb R}}
\newcommand{\N}{{\Bbb N}}

\newcommand{\pa}{\partial}

\makeatletter

\@addtoreset{equation}{section}
\makeatother

\usepackage{fancyhdr}
\title{\Large\sf A higher speed type II blowup for the five dimensional energy critical heat equation}
\author{Junichi Harada
\\{\small Faculty of Education and Human Studies, Akita University}
\\[1mm]{\small email: harada-j@math.akita-u.ac.jp}}

\begin{document}
\maketitle
\thispagestyle{empty}

\begin{abstract}
This paper is concerned with blow-up solutions of
the five dimensional energy critical heat equation $u_t=\Delta u+|u|^\frac{4}{3}u$.
A goal of this paper is to show the existence of type II blowup solutions
which behave as $\|u(t)\|_\infty\sim(T-t)^{-3k}$ ($k=2,3,\cdots$).
Our solutions are the same one formally derived by
Filippas, Herrero and Vel\'azquez \cite{Filippas}.
\end{abstract}


\section{Introduction}
This paper is concerned with blowup solutions for the semilinear heat equation.
 \begin{equation}\label{1.1}
 u_t = \Delta u+|u|^{p-1}u
 \qquad
 \text{in }
 \R^n\times(0,T).
 \end{equation}
This problem is a simple model of nonlinear diffusion problems.
Various complex and interesting phenomena have been found for this 30 years.
Our concern in this paper is a blowup caused by a concentration.
A local solvability of this problem is well understood,
and a blow up occurs at $t=T$ if $\limsup_{t\to T}\|u(t)\|_\infty=\infty$.
For a blowup solution,
the blowup is called type I if
$\limsup_{t\to T}(T-t)^\frac{1}{p-1}\|u(t)\|_\infty<\infty$,
and type II if $\limsup_{t\to T}(T-t)^\frac{1}{p-1}\|u(t)\|_\infty=\infty$.
A typical type I blowup solution is given by
 \begin{equation}\label{1.2}
 w(x,t) = (p-1)^{-\frac{1}{p-1}}(T-t)^{-\frac{1}{p-1}}.
 \end{equation}
In the study of blowup problems,
there are two important critical values of $p$ defined by
 \[
 p_{\text{S}}=\frac{n+2}{n-2},
 \hspace{10mm}
 p_{\text{JL}}=
 \begin{cases}
 \infty & \text{if } n\leq10,
 \\ \dis
 1+\frac{4}{n-4-2\sqrt{n-1}} & \text{if } n\geq11.
 \end{cases}
 \]
For the case $1<p<p_\text{S}$,
it is well known that every blowup solution is locally approximated by \eqref{1.2},
namely type I (see \cite{Giga}).
On the other hand,
for the case $p\geq p_S$,
different types of blowup behavior are observed.
A type II blowup solution is first discovered by Herrero and Vel\'azquez
\cite{Herrero1,Herrero2} (see also \cite{Mizoguchi}).
They construct type II blowup solutions with the exact blowup rates
for the case $p>p_{\text{JL}}$.
Very recently Seki \cite{Seki} proves the existence of type II blowup solutions
for the case $p=p_{\text{JL}}$.
In the middle range of $p\in(p_\text{S},p_\text{JL})$,
Matano and Merl \cite{Matano} exclude the occurrence of a type II blowup
under a radial setting.
Another example of a type II blowup is found by Filippas, Herrero and Vel\'azquez
\cite{Filippas} for the energy critical case $p=p_{\text{S}}$.
They formally obtain a type II blowup
by using the matched asymptotic expansion approach.
The blowup rate of their solutions are given by
(the blowup rates for $n=5$ p. 2971 and $n=6$ p. 2972 in \cite{Filippas}
seem to be incorrect by a trivial miscalculation)
\begin{equation}\label{1.3}
 \|u(t)\|_\infty
 \sim
 \begin{cases}
 (T-t)^{-k} & n=3,
 \\
 (T-t)^{-k}|\log(T-t)|^\frac{2k}{2k-1} & n=4,
 \\
 (T-t)^{-3k} & n=5 \quad (\text{corrected}),
 \\
 (T-t)^{-\frac{5}{2}}|\log(T-t)|^{\frac{15}{4}} & n=6 \quad (\text{corrected}),
 \end{cases}
\end{equation}
where $k=1,2,3,\cdots$.
As for a higher dimensional case $n\geq7$,
the possibility of a type II blowup near the ground states is ruled out
by Collot, Merle and Rapha\"el \cite{Collot}.
The first rigorous proof of the existence of a type II blowup
for the energy critical case is given by Schweyer \cite{Schweyer}.
He constructs a type II blowup solution for $n=4$
by adapting the energy method developed
in the study of geometrical dispersive problems
(\cite{Raphael,Merle-R-R}) to the problem \eqref{1.1}.
The blowup rate of his solution coincides with $k=1$ in \eqref{1.3}.
Very recently Cort\'azar, del Pino and Musso \cite{delPino3} obtain a type II blowup
for $n=5$ with the same blowup rate as $k=1$ in \eqref{1.3}.
They apply so-called the inner-outer gluing method
developed in \cite{Cortazar,Davila,delPino1}.
Furthermore
a new type of type II blowup not listed in \eqref{1.3} is found
by del Pino, Musso and Wei \cite{delPino2} for $n\geq7$.
In this paper,
we prove the existence of type II blowup solutions
for $n=5$ with a higher blowup speed.
The solutions constructed here give the first example for $k\geq2$ in \eqref{1.3}.

\section{Main result}
Let ${\sf Q}_\lambda(x)$ be the positive radial stationary solution given by
 \[
 {\sf Q}_\lambda(x)
 =
 \lambda^{-\frac{n-2}{2}
 }\left( 1+\frac{1}{n(n-2)}\frac{|x|^2}{\lambda^2} \right)^{-\frac{n-2}{2}}
 \qquad
 (\text{ground state}).
 \]
\begin{thm}\label{Thm1}
 Let $n=5$ and $p=p_\text{S}$.
 For any integer $l\geq1$ and any two constants $A>0$, $\kappa\in(0,1)$,
 there exist $T>0$ and a radial solution
 $u(x,t)\in C(\R^5\times[0,T))\cap C^{2,1}(\R^5\times(0,T))$
 of \eqref{1.1}  such that
 \[
 u(x,t)
 =
 {\sf Q}_{\lambda(t)}(x)+v(x,t),
 \]
 where $\lambda(t)$ and $v(x,t)$ satisfy
 \[
 \left| \lambda(t)-A(T-t)^{2l+2} \right|
 <
 \kappa A(T-t)^{2l+2}
 \qquad\text{\rm and}\qquad
 v(x,t)\in L^\infty(\R^5\times(0,T)).
 \]
\end{thm}
\begin{rem}
The blowup rate of the solution obtained in Theorem \ref{Thm1} is given by
 \[
 \|u(t)\|_\infty
 \sim
 \|{\sf Q}_{\lambda(t)}\|_\infty
 \sim
 \lambda(t)^{-\frac{n-2}{2}}
 \sim
 (T-t)^{-3(l+1)}
 \qquad (l=1,2,3,\cdots).
 \]
 This solution gives an example for $k=l+1\geq2$ in \eqref{1.3}.
\end{rem}
\begin{rem}
 Our strategy is based on so-called the inner-outer gluing method used
 in \cite{delPino3}.
 We look for a solution of the form
 \[
 u(x,t)
 =
 {\sf Q}_{\lambda(t)}
 +
 \underbrace{
 \Theta(x,t)
 +
 \lambda^{-\frac{n-2}{2}}\epsilon(y,t)+w(x,t)
 }_{=v(x,t)},
 \qquad y=\frac{x}{\lambda}.
 \]
 The function $\Theta(x,t)$ is a particular solution of $\Theta_t=\Delta_x\Theta$.
 In \cite{delPino3},
 they consider a simpler case,
 where $\Theta(x,t)$ is chosen to be a constant function.
 In this paper,
 we try other types of a particular solution $\Theta(x,t)$ satisfying
 $\Theta(x,t)\sim(T-t)^l$ for $|x|\sim\sqrt{T-t}$.
 This contributes to the blowup rate.
 Generally the function $\Theta(x,t)$ can not be chosen arbitrarily,
 since it must satisfy a certain matching condition (see Section \ref{Sec4}).
 The functions $\epsilon(y,t)$ and $w(x,t)$ describe the behavior
 in the inner region $|x|\sim\lambda(t)$ and
 in the self-similar region $|x|\sim\sqrt{T-t}$ respectively.
 Since the behavior of the inner solution $\epsilon(y,t)$ is
 almost the same as that of \cite{delPino3},
 it can be treated in the same manner.
 A main part of this paper is to handle the outer solution $w(x,t)$.
 To derive an appropriate decay estimate of $w(x,t)$,
 we borrow techniques from \cite{Herrero1,Herrero2}
 (see also \cite{Mizoguchi,Seki}),
 where they treat a different type of outer solutions.
\end{rem}

\section{Preliminary}
Throughout this paper,
$\chi(\xi)\in C^\infty(\R)$ stands for a standard cut off function satisfying
\[
 \chi(\xi)=
 \begin{cases}
 1 & \text{if } \xi<1,\\
 0 & \text{if } \xi>2.
 \end{cases}
\]
Furthermore we write
 \[
 f(u)=|u|^{p-1}u,
 \qquad
 p=\frac{n+2}{n-2}.
 \]

\subsection{Linearization aound the ground state}
\label{Sec3.1}
Let us consider the eigenvalue problem related to a linearization around
the ground state ${\sf Q}(y)={\sf Q}_\lambda(y)|_{\lambda=1}$.
 \begin{equation}\label{3.1}
 -H_y\psi=\mu\psi \qquad \text{in } \R^n,
 \end{equation}
where the operator $H_y$ is define by
 \[
 H_y=\Delta_y+V(y),
 \qquad
 V(y)=f'({\sf Q}(y))=p{\sf Q}(y)^{p-1}.
 \]
We recall that the operator $H_y$ has a negative eigenvalue $\mu_1<0$ and a zero eigenvalue.
We denote by $\psi_1(r)$ a positive radial eigenfunction
associated to the negative eigenvalue with $\psi_1(0)=1$.
Furthermore
there exists $C>0$ such that
 \[
 \psi_1(r)
 <
 C\left( 1+r \right)^{-\frac{n-1}{2}}
 e^{-\sqrt{|\mu_1|}\,r}.
 \]
The eigenfunction associated to a zero eigenvalue is explicitly given by
 \[
 \Lambda_y{\sf Q}(y)=\left( \frac{n-2}{2}+y\cdot\nabla_y \right){\sf Q}(y).
 \]

\subsection{Purterbated linearized problem}
\label{Sec3.2}
We next consider the eigenvalue problem \eqref{3.1} in a bounded but very large domain.
 \begin{equation}\label{3.2}
 \begin{cases}
 -H_y\psi=\mu\psi & \text{in } B_R,
 \\
 \psi=0 & \text{on } \pa B_R,
 \\
 \psi \text{ is raidal}.
 \end{cases}
 \end{equation}
We denote the $i$th eigenvalue of \eqref{3.2} by $\mu_i^{(R)}$
and the associated eigenfunction by $\psi_i^{(R)}$.
We normalize $\psi_i^{(R)}(r)$ as  $\psi_i^{(R)}(0)=1$.
Most of the lemmas stated in this subsection are proved in Section 7 \cite{Cortazar}.
However for the sake of convenience,
we give the proofs.
Throughout this subsection,
we write
 \[
 k_1\lesssim k_2
 \qquad
 (k_1,k_2>0)
 \]
if there is a universal constant $c>0$ independent of $R$ such that $k_1<ck_2$.
This definitions will be changed slightly in Section \ref{Sec5.3}. 
 \begin{lem}\label{Lem3.1}
 It holds that for any $R>0$
 \[
 0
 <
 \psi_1^{(R)}(r)
 \lesssim
 \left( 1+r \right)^{-\frac{n-1}{2}}
 e^{-\sqrt{|\mu_1|}\,r}
 \qquad\text{\rm for } r\in(0,R).
 \]
 \end{lem}
 \begin{proof}
It is enough to prove $\psi_1^{(R)}(r)<2\psi_1(r)$.
We prove by contradiction.
Suppose that there exists $r_1\in(0,R)$ such that
 $\psi_1^{(R)}(r)<2\psi_1(r)$ for $r<r_1$ and
 $\psi_1^{(R)}(r_1)=2\psi_1(r_1)$.
 We now define $\psi_0\in H_0^1(B_{R})$ as
 \[
 \psi_0(r)=
 \begin{cases}
 2\psi_1(r)-\psi_1^{(R)}(r) & \text{for } r<r_1,
 \\
 0  & \text{for } r_1<r<R.
 \end{cases}
 \]
By the monotonic dependence of the eigenvalue with respect to the domain,
it holds that $\mu_1^{(R)}>\mu_1$.
Therefore we get
 \begin{align*}
 \int_{B_R}\left( |\nabla_y\psi_0|^2-V\psi_0^2 \right)dy
 &=
 \mu_1^{(R)}\int_{B_R}\psi_0^2dy+(\mu_1-\mu_1^{(R)})\int_{B_R}2\psi_1\psi_0dy
 \\
 &<
 \mu_1^{(R)}\int_{B_R}\psi_0^2dy.
 \end{align*}
However this contradicts to characterization of $\mu_1^{(R)}$.
The proof is completed.
 \end{proof}
 \begin{lem}\label{Lem3.2}
 There exists $R_1>0$ such that if $R>R_1$
 \[
 |\psi_2^{(R)}(r)|
 \lesssim
 \left( 1+r \right)^{-(n-2)}
 \qquad\text{\rm for } r\in(0,R).
 \]
 \end{lem}
 \begin{proof}
Put $\psi_2(r)=\Lambda_y{\sf Q}(r)/\Lambda_y{\sf Q}(0)$.
It is clear that $\psi_2$ gives the eigenfunction of \eqref{3.1}
associated to a zero eigenvalue and $\psi_2(0)=1$.
Let $r_0$ and $r_0^{(R)}$ be the unique zero of $\psi_2(r)$ and $\psi_2^{(R)}$
respectively.
The Sturm comparison principle implies that $\mu_2^{(R)}>0$ and $r_0^{(R)}<r_0$.
We easily see that
$\lim_{R\to\infty}\mu_2^{(R)}=0$
and
$\lim_{R\to\infty}\psi_2^{(R)}(r)=\psi_2(r)$ locally uniformly in $r\in[0,\infty)$.
Therefore
there exists $R_1>0$ such that $\psi_2(r_0+1)<\frac{1}{2}\psi_2^{(R)}(r_0+1)<0$
if $R>R_1$.
We now define $r_1>r_0$ as
 \[
 r_1
 =
 \sup\left\{ r\in(r_0,r_0+1);\ \psi_2(r)=\frac{1}{2}\psi_2^{(R)}(r) \right\}.
 \]
From this definition,
we immediately see that $\psi_2(r_1)=\frac{1}{2}\psi_2^{(R)}(r_1)$ and
 \[
 \psi_2(r)<\frac{1}{2}\psi_2^{(R)}(r)
 \quad\text{for }
 r_1<r<r_0+1.
 \]
We now suppose that there exists $r_2>r_1$ such that
$\psi_2(r_2)=\frac{1}{2}\psi_2^{(R)}(r_2)$ and
 \[
 \psi_2(r)<\frac{1}{2}\psi_2^{(R)}(r)
 \quad\text{for }
 r_1<r<r_2.
 \]
However this contradicts the Sturm comparison principle.
Therefore we obtain
 \[
 \psi_2(r)<\frac{1}{2}\psi_2^{(R)}(r)<0
 \quad\text{for }
 r_0+1<r<R.
 \]
Since $|\psi_2(r)|\lesssim(1+r)^{-(n-2)}$,
we complete the proof.
 \end{proof}
 \begin{lem}[Lemma 7.2 \cite{Cortazar}]\label{Lem3.3}
 Let $n\geq5$.
 There exists $R_2>0$ such that if $R>R_2$
 \[
 \mu_2^{(R)}
 \gtrsim
 R^{-(n-2)}.
 \]
 \end{lem}
 \begin{proof}
 We recall that $Z_1(r)=\Lambda_y{\sf Q}(r)$ gives a solution of $H_yZ=0$.
 Let $Z_2(r)=\Gamma(r)$ be another independent solution of $H_yZ=0$.
 Since $\psi_2^{(R)}(r)$ is a solution of \eqref{3.2} with $\mu=\mu_2^{(R)}$,
 it is expressed as
 \begin{align*}
 \psi_2^{(R)}(r)
 &=
 kZ_2(r)\int_0^r\mu_2^{(R)}\psi_2^{(R)}(r')Z_1(r')r'^{n-1}dr'
 +
 kZ_1(r)\int_r^R\mu_2^{(R)}\psi_2^{(R)}(r')Z_2(r')r'^{n-1}dr'
 \\
 &\qquad
 -
 k\frac{Z_2(R)}{Z_1(R)}Z_1(r)
 \int_0^R\mu_2^{(R)}\psi_2^{(R)}(r')Z_1(r')r'^{n-1}dr'
 \\
 &=:
 \mu_2^{(R)}(A_1+A_2-A_3),
 \end{align*}
 where $k$ is a constant depending on $Z_1(r)$ and $Z_2(r)$.
 Since $|Z_2(r)|\lesssim1+r^{-(n-2)}$,
 we easily see that
 \begin{align*}
 \|A_1\|_{L^2(B_R)}
 &\lesssim
 \|\psi_2^{(R)}\|_{L^\infty(B_1)}
 \|Z_1\|_{L^\infty(B_1)}
 +
 R^\frac{n}{2}
 \|\psi_2^{(R)}\|_{L^2(B_R)}
 \|Z_1\|_{L^2(B_R)},
 \\
 \|A_2\|_{L^2(B_R)}
 &\lesssim
 \|Z_1\|_{L^2(B_R)}
 \|\psi_2^{(R)}\|_{L^\infty(B_1)}
 +
 R^\frac{n}{2}
 \|Z_1\|_{L^2(B_R)}
 \|\psi_2^{(R)}\|_{L^2(B_R)},
 \\
 \|A_3\|_{L^2(B_R)}
 &\lesssim
 \left| \frac{Z_2(R)}{Z_1(R)} \right|
 \|Z_1\|_{L^2(B_R)}^2
 \|\psi_2^{(R)}\|_{L^2(B_R)}.
 \end{align*}
Since $\lim_{R\to\infty}\psi_2^{(R)}(r)=\psi_2(r)$ uniformly in $r\in[0,1]$
(see Lemma \ref{Lem3.2}),
it follows that $\|\psi_2^{(R)}\|_{L^\infty(B_1)}\lesssim1$.
Therefore when $n\geq5$,
we get from Lemma \ref{Lem3.2} that
 \begin{align*}
 \|\psi_2^{(R)}\|_{L^2(B_R)}
 &\lesssim
 \mu_2^{(R)}
 \left(
 1+R^\frac{n}{2}+R^{n-2}
 \right).
 \end{align*}
Since
$\lim_{R\to\infty}\|\psi_2^{(R)}\|_{L^2(B_R)}=\|\Lambda_y{\sf Q}\|_{L_2(\R^n)}$
if $n\geq5$ (see Lemma \ref{Lem3.2}),
we complete the proof.
 \end{proof}

\subsection{Behavior of the Laplace equation wiht a purterbation term}
\label{Sec3.3}
Consider a radial solution of
 \[
 \Delta_yp+(1-\chi_M)V(y)p=0 \qquad \text{in } \R^n,
 \]
 where $\chi_M(y)=\chi(\frac{|y|}{M})$.
 Let $p_M(r)$ be a radial solution of this problem satisfying
 $p_M(r)=1$ for $r<M$.
 \begin{lem}[see proof of Lemma 7.3 \cite{Cortazar}]\label{Lem3.4}
 There exist $k\in(0,1)$ and $M_1>0$ such that if $M>M_1$
 \[
 k<p_M(r)\leq1
 \qquad\text{\rm for}\ r\in(0,\infty).
 \]
 \end{lem}
\begin{proof}
Since $V(r)\sim\frac{1}{r^4}$ for $r>1$,
we can take $M_1>0$ such that
 \[
 V(r)<\frac{n-3}{r^2}
 \qquad\text{for } r>M_1.
 \]
Let $\bar p(r)=\frac{M_1^{n-3}}{r^{n-3}}$.
It satisfies
 \[
 \Delta_y\bar p-\frac{n-3}{r^2}\bar p=0.
 \]
Therefore
by the Sturm comparison principle,
it holds that
 \begin{equation}\label{3.3}
 p_{M_1}(r)
 >
 \bar p(r)
 =
 \frac{M_1^{n-3}}{r^{n-3}}
 \qquad\text{for } r>M_1.
 \end{equation}
Let $Z_1(r)$ and $Z_2(r)$ be given in the proof of Lemma \ref{Lem3.3}.
Since $p_{M_1}(r)$ satisfies $H_yp_{M_1}=0$ for $r>2M_1$,
there exist two constants $c_1,c_2\in\R$ such that
$p_{M_1}(r)=c_1Z_1(r)+c_2Z_2(r)$ for $r>2M_1$.
We recall that $\lim_{r\to\infty}Z_2(r)=\alpha\not=0$.
Since $p_{M_1}(r)$ satisfies \eqref{3.3},
$c_2\alpha$ must be positive.
Therefore
it holds that $k:=\inf_{r>0}p_{M_1}(r)>0$.
For the case $M>M_1$,
by the Sturm comparison principle,
we conclude $p_M(r)\geq p_{M_1}(r)\geq k$ for $r>0$.
Since $p_M(r)$ is positive,
we easily check that $\pa_rp_M(r)\leq0$,
which implies $p_M(r)<1$.
The proof is completed.
\end{proof}

\subsection{The Schauder estimate for parabolic equations}
\label{Sec3.4}
Put $Q=B_1\times(0,1)$.
For $\alpha\in(0,1)$,
we define the H\"older norm.
 \[
 \|u\|_{C^{\alpha}(Q)}
 :=
 \|u\|_{L^\infty(Q)}
 +
 [u]_{C^\alpha(Q)},
 \qquad
 [u]_{C^\alpha(Q)}
 :=
 \sup_{(x_1,t_1),(x_2,t_2)\in Q}
 \frac{|u(x_1,t_1)-u(x_2,t_2)|}{|x_1-x_2|^\alpha+|t_1-t_2|^\frac{\alpha}{2}}.
 \]
We recall the local H\"oder estimate for parabolic equations.
 \begin{lem}[Theorem 4.8 p. 56 \cite{Lieberman}]\label{Lem3.5}
 Let $\alpha\in(0,1)$ and $V(x,t)\in L^\infty(Q)$ satisfy $\|V\|_{L^\infty(Q)}<M$.
 There exists $c_M>0$ such that
 if $u(x,t),\nabla_xu(x,t)\in C^\alpha(Q)$ and $u(x,t)$ satisfies
 \[
 u_t=\Delta_xu+V(x,t)u+f(x,t) \qquad \text{\rm in}\ B_1\times(0,1),
 \]
 then
 \[
 \|u\|_{C^{\alpha}(B_\frac{1}{2}\times(\frac{1}{2},1))}
 +
 \|\nabla_x u\|_{C^{\alpha}(B_\frac{1}{2}\times(\frac{1}{2},1))}
 <
 c_M\left( \|u\|_{L^\infty(Q)}+\|f\|_{L^\infty(Q)} \right).
 \]
 \end{lem}

\subsection{Local behavior of the heat equation}
\label{Sec3.5}
Consider the heat equation
 \begin{equation}\label{3.4}
 u_t=\Delta u \qquad\text{in } \R^n\times(0,\infty).
 \end{equation}
To describe the local behavior of solutions,
we use self-similar variables.
 \[
 \theta(z,\tau)=u(x,t), \qquad z=\frac{x}{\sqrt{1-t}},\quad 1-t=e^{-\tau}.
 \]
This function $\theta(z,\tau)$ solves
 \begin{equation}\label{3.5}
 \theta_\tau=A_z\theta
 \qquad\text{in } \R^n\times(0,\infty),
 \end{equation}
where $A_z=\Delta_z-\frac{z}{2}\cdot\nabla_z$.
We define the associated wighted $L^2$ space by
 \begin{align*}
 L_\rho^2(\R^n)
 :=
 \{f\in L_\text{loc}^2(\R^n);\ \|f\|_\rho<\infty\},
 \qquad
 \|f\|_\rho^2=\int_{\R^n}f(z)^2\rho(z)dz,
 \qquad
 \rho(z)=e^{-\frac{|z|^2}{4}}.
 \end{align*}
The inner product is denoted by
 \[
 (f_1,f_2)_\rho==\int_{\R^n}f_1(z)f_2(z)\rho(z)dz.
 \]
Consider the eigenvalue problem
 \[
 -\Delta e_i+\frac{z}{2}\cdot e_i=\lambda_le_i
 \qquad\text{in } L_{\rho,\text{rad}}^2(\R^n).
 \]
It is known that
 \begin{itemize}
 \item $\lambda_i=i$ \ ($i=0,1,2,\cdots$) and
 \item $e_i(z)\in L_{\rho,\text{rad}}^2(\R^n)$ is the $2l$th-degree polynomial.
 \end{itemize}
We normalize the eigenfunction $e_i(z)$ as $e_i(0)=1$,
which implies
 \[
 e_i(z)=1+a_1|z|^2+a_2|z|^4+\cdots+a_l|z|^{2l}.
 \]
The function
 \begin{equation}\label{3.6}
 \Theta_i(x,t)=A(T-t)^ie_i(z),
 \qquad z=\frac{x}{\sqrt{T-t}}
 \qquad (i=0,1,2,\cdots)
 \end{equation}
gives a solution of \eqref{3.4}.
This function plays a crucial role in our argument.
The constant $A$ is chosen to be $A=-1$ later.
We denote by $e^{A_z\tau}\theta_0$ a solution of \eqref{3.5} with the initial date
$\theta_0$ for $\tau=0$.
This is expressed by
 \[
 e^{A_z\tau}\theta_0
 =
 \frac{c_n}{(1-e^{-\tau})^\frac{n}{2}}
 \int_{\R^n}
 e^\frac{|e^{-\frac{\tau}{2}}z-\xi|^2}{4(1-e^{-\tau})}
 \theta_0(\xi)
 d\xi.
 \]
By using this formula,
we can obtain the following parabolic estimate for \eqref{3.5}
(see proof of Lemma 2.2 in \cite{Harada2}).
 \begin{lem}\label{Lem3.6}
 There exists $C>0$ such that
 \[
 |(e^{A_z\tau}\theta_0)(z)|
 <
 C
 \frac{ e^\frac{e^{-\tau}|z|^2}{4(1+e^{-\tau})}}{(1-e^{-\tau})^{\frac{n}{4}}}
 \|\theta_0\|_\rho
 \qquad
 \text{\rm for}\ (z,\tau)\in\R^n\times(0,\infty).
 \]
 \end{lem}
 We prepare another type of parabolic estimates given in \cite{Harada}.
 \begin{lem}[Lemma 2.2 in \cite{Harada})]\label{Lem3.7}
 For any $l\in\N$,
 there exists $C_l>0$ such that if $\theta_0=|z|^{2l}$
 \[
 |(e^{A_z\tau}\theta_0)(z)|
 <
 C_l\left( 1+e^{-l\tau}|z|^{2l} \right)
 \qquad
 \text{\rm for}\ (z,\tau)\in\R^n\times(0,\infty).
 \]
 \end{lem}
We next consider the nonhomogeneous heat equation.
 \begin{equation}\label{3.7}
 \begin{cases}
 \dis
 \Phi_\tau
 =
 A_z\Phi+\frac{e^{-\tau}}{{\lambda(\tau)}^2}\frac{1}{1+|y|^{2+a}}
 & \text{in } \R^n\times(\tau_1,\infty),
 \\
 \Phi=0
 & \text{for } \tau=\tau_1,
 \end{cases}
 \qquad y=\frac{e^{-\frac{\tau}{2}}}{\lambda(\tau)}z.
 \end{equation}
 \begin{lem}[Lemma 4.2 p8 \cite{delPino3}]\label{Lem3.8}
 Let $a>0$, $\gamma>\frac{1}{2}$ and $T_1=e^{-\tau_1}$.
 If $\lambda(\tau)$ satisfies
 \[
 k_1e^{-\gamma\tau}< \lambda <k_2e^{-\gamma\tau}
 \qquad\text{\rm and}\qquad
 \left|
 \frac{d\lambda}{d\tau} \right|<k_3\lambda,
 \]
 there exists $C>0$ depending on $a,\gamma,k_1,k_2,k_3$ such that
 a solution $\Phi(z,\tau)$ of {\rm\eqref{3.7}} satisfies
 \[
 |\Phi(z,\tau)|
 <
 C\left( \frac{1}{1+|y|^a}+T_1^{(\gamma-\frac{1}{2})a} \right),
 \qquad
 y=\frac{e^{-\frac{\tau}{2}}}{\lambda}z.
 \]
 \end{lem}
\begin{proof}
We change variable.
 \[
 \psi(x,t)=\Phi(e^\frac{\tau}{2}x,\tau),
 \qquad T_1-t=-e^{-\tau}.
 \]
The function $\psi(x,t)$ solves
 \[
 \begin{cases}
 \dis
 \psi_t
 =
 \Delta_x\psi
 +
 \frac{1}{\lambda^2}\frac{1}{1+|y|^{2+a}}
 & \text{in } \R^n\times(0,T_1),
 \\
 \psi=0
 & \text{for } t=0,
 \end{cases}
 \qquad y=\frac{x}{\lambda}.
 \]
From Lemma 2.2 in \cite{delPino3},
there exists $C>0$ such that
 \[
 |\psi(x,t)|
 <
 C\left( \frac{1}{1+|y|^a}+T_1^{(\gamma-\frac{1}{2})a} \right).
 \]
The proof is completed.
\end{proof}

\section{Formal derivation of blowup speed}
\label{Sec4}
We fix $l\in\N$.
Throughout this paper,
we write (see \eqref{3.6})
 \[
 \Theta(x,t)=\Theta_l(x,t)=A(T-t)^le_l(z),
 \qquad
 z=\frac{x}{\sqrt{T-t}}.
 \]
We look for solutions of the form
 \begin{equation}\label{4.1}
 u(x,t)={\sf Q}_{\lambda(t)}(x)+\Theta(x,t)+v(x,t),
 \end{equation}
where $v(x,t)$ is a remainder term.
A function $v(x,t)$ satisfies
 \[
 v_t
 =
 \Delta v
 +
 f'({\sf Q}_\lambda)(\Theta+v)
 +
 {\sf N}(v)
 +
 \frac{\lambda_t}{\lambda^\frac{n}{2}}\Lambda_y{\sf Q}(y),
 \qquad
 y=\frac{x}{\lambda}.
 \]
The nonlinear term ${\sf N}(v)$ is defined by
\[
 {\sf N}(v)
 =
 f({\sf Q}_\lambda+\Theta+v)-f({\sf Q}_\lambda)-f'({\sf Q}_\lambda)(\Theta+v).
 \]
We write $v(x,t)$ as
 \[
 v(x,t)
 =
 \lambda^{-\frac{n-2}{2}}\epsilon(y,t),
 \qquad
 y=\frac{x}{\lambda}.
 \]
The relation \eqref{4.1} is rewritten as
 \[
 u(x,t)={\sf Q}_{\lambda(t)}(x)+\Theta(x,t)+\epsilon_{\lambda(t)}(x,t).
 \]
Under this setting,
it is natural to assume that
 \begin{equation}\label{4.2}
 |\epsilon(y,s)|\ll{\sf Q}(y).
 \end{equation}
The function $\epsilon(y,t)$ satisfies
 \begin{align*}
 \frac{\epsilon_t}{\lambda^{\frac{n-2}{2}}}
 =
 \frac{H_y\epsilon}{\lambda^{\frac{n+2}{2}}}
 +
 \frac{V(y)}{\lambda^2}\Theta(x,t)
 +
 {\sf N}(v)
 +
 \frac{\lambda_t}{\lambda^{\frac{n}{2}}}\Lambda_y{\sf Q}(y)
 +
 \frac{\lambda_t}{\lambda^{\frac{n}{2}}}\Lambda_y\epsilon.
 \end{align*}
Neglecting ${\sf N}(v)$ and assuming \eqref{4.2},
we obtain
 \begin{align*}
 \lambda^2\epsilon_t
 \sim
 H_y\epsilon
 +
 \lambda^\frac{n-2}{2}V(y)\Theta(x,t)
 +
 \lambda\lambda_t\Lambda_y{\sf Q}(y).
 \end{align*}
Since $x=\lambda(t)y$ and $\lim_{t\to T}\lambda(t)=0$,
we here replace $\Theta(x,t)$ by $\Theta(0,t)$.
 \begin{align*}
 \lambda^2\epsilon_t
 \sim
 H_y\epsilon
 +
 \lambda^\frac{n-2}{2}V(y)\Theta(0,t)
 +
 \lambda\lambda_t\Lambda_y{\sf Q}(y).
 \end{align*}
We take the inner product $(\cdot,\Lambda_y{\sf Q})_{L_y^2(\R^n)}$ to get
 \begin{align*}
 \lambda^2(\epsilon_t,\Lambda_y{\sf Q})_2
 &\sim
 \lambda^\frac{n-2}{2}\Theta(0,t)(V,\Lambda_y{\sf Q})_2
 +
 \lambda\lambda_t\|\Lambda_y{\sf Q}\|_2^2.
 \end{align*}
In addition to \eqref{4.2},
we assume that the left-hand side is negligible in the relation.
Since $\Theta(0,t)=A(T-t)^l$,
we obtain
 \[
 0
 \sim
 \lambda^\frac{n-2}{2}A(T-t)^l(V,\Lambda_y{\sf Q})_2
 +
 \lambda\lambda_t\|\Lambda_y{\sf Q}\|_2^2.
 \]
By a direct calculation,
the first term is computed as
 \begin{align}\label{4.3}
 (V,\Lambda_y{\sf Q})_2
 &=
 (p{\sf Q}^{p-1},\Lambda_y{\sf Q})_2
 =
 \left.
 \frac{d}{d\lambda}\int_{\R^n}{\sf Q}_\frac{1}{\lambda}^pdy
 \right|_{\lambda=1}
 \nonumber
 \\
 &=
 \left.
 \frac{d}{d\lambda}\lambda^{-\frac{n-2}{2}}\|{\sf Q}\|_p^p
 \right|_{\lambda=1}
 =
 -\frac{n-2}{2}\|{\sf Q}\|_p^p.
 \end{align}
Therefore we obtain a differential equation for $\lambda(t)$.
 \[
 \frac{\lambda_t}{\lambda^{\frac{n-4}{2}}}
 \sim
 \frac{A(n-2)}{2}\frac{\|{\sf Q}\|_p^p}{\|\Lambda_y{\sf Q}\|_2^2}(T-t)^l.
 \]
Since $\lambda(t)$ must be positive and $\lim_{t\to T}\lambda(t)=0$,
the constant $A$ must be negative.
We finally obtain
 \[
 \lambda(t)
 \sim
 \left(
 \frac{-A(6-n)(n-2)}{4(l+1)}\frac{\|{\sf Q}\|_p^p}{\|\Lambda_y{\sf Q}\|_2^2}
 \right)^\frac{2}{6-n}
 (T-t)^\frac{2(l+1)}{6-n}.
 \]
This is the desired result.
From now on,
we choose
\[
 A=-1.
\]

\section{Formulation}
\label{Sec5}
In this section,
we set up our problem as in the proof of Theorem 1 \cite{delPino3}.
To justify the argument in Section \ref{Sec4},
we need several corrections.

\subsection{Setting}
\label{Sec5.1}
We look for solutions of the form
 \[
 u(x,t)={\sf Q}_{\lambda(t)}(x)+\Theta(x,t)\chi_\text{out}+v(x,t),
 \]
where $\chi_\text{out}$ is a cut off function defined by
 \[
 \chi_\text{out}=\chi\Bigl( (T-t)^B|z| \Bigr),
 \qquad
 z=\frac{x}{\sqrt{T-t}},
 \qquad
 B=\frac{l+\frac{1}{2}}{2l+2}.
 \]
A function $v(x,t)$ satisfies
 \begin{equation}\label{5.1}
 v_t
 =
 \Delta v
 +
 f'({\sf Q}_\lambda)(\Theta\chi_\text{out}+v)
 +
 {\sf N}(v)
 +
 \frac{\lambda_t}{\lambda^\frac{n}{2}}\Lambda_y{\sf Q}(y)
 +
 h_\text{out},
 \end{equation}
where
 \begin{equation}\label{5.2}
 h_\text{out}
 =
 2\nabla_x\Theta\cdot\nabla_x\chi_\text{out}
 +
 \Theta\Delta_x\chi_\text{out}
 -
 \Theta\pa_t\chi_\text{out}.
 \end{equation}
We decompose $v(x,t)$ as
 \begin{equation}\label{5.3}
 v(x,t)
 =
 \lambda^{-\frac{n-2}{2}}\epsilon(y,t)\chi_\text{in}
 +
 w(x,t),
 \qquad
 y=\frac{x}{\lambda}.
 \end{equation}
The function $\epsilon(y,t)$ is defined on $(y,t)\in B_{2R}\times(0,T)$ and
$\chi_\text{in}=\chi(\frac{|y|}{R})$.
Plugging this into \eqref{5.1},
we get
 \begin{align*}
 \frac{\epsilon_t}{\lambda^{\frac{n-2}{2}}}
 \chi_\text{in}
 +
 w_t
 &=
 \frac{H_y\epsilon}{\lambda^{\frac{n+2}{2}}}
 \chi_\text{in}
 +
 \Delta w
 +
 \frac{V(y)}{\lambda^2}(\Theta\chi_\text{out}+v)
 +
 {\sf N}(v)
 +
 \frac{\lambda_t}{\lambda^{\frac{n}{2}}}\Lambda_y{\sf Q}(y)
 +
 h_\text{out}+h_\text{in},
 \end{align*}
 where
 \begin{align}\label{5.4}
 h_\text{in}
 =
 \frac{1}{\lambda^{\frac{n+2}{2}}}
 \left(
 2\nabla_y\epsilon\cdot\nabla_y\chi_\text{in}+\epsilon\Delta_y\chi_\text{in}
 \right)
 +
 \frac{\lambda_t}{\lambda^{\frac{n}{2}}}\Lambda_y\epsilon\chi_\text{in}
 -
 \frac{1}{\lambda^\frac{n-2}{2}}\epsilon\pa_t\chi_\text{in}.
 \end{align}
We introduce a parabolic system of $(\epsilon(y,t),w(x,t))$.
 \begin{equation}\label{5.5}
 \begin{cases}
 \lambda^2\epsilon_t
 =
 H_y\epsilon+G_\text{in}(\lambda,w)
 &
 \text{in } B_{2R}\times(0,T),
 \\
 w_t
 =
 \Delta_xw+G_\text{out}(\lambda,w,\epsilon)
 &
 \text{in } \R^5\times(0,T),
 \end{cases}
 \end{equation}
where
 \begin{align}
 \label{5.6}
 G_\text{in}(\lambda,w)
 &=
 \lambda^\frac{n-2}{2}
 \Theta(x,t)V
 +
 \lambda^\frac{n-2}{2}w(x,t)V
 +
 \lambda\lambda_t\Lambda_y{\sf Q},
 \\[1mm]
 \label{5.7}
 G_\text{out}(\lambda,w,\epsilon)
 &=
 h_\text{out}+h_\text{in}
 +
 \frac{1}{\lambda^2}(1-\chi_\text{in})V(y)(\Theta\chi_\text{out}+w)
 +
 \frac{\lambda_t}{\lambda^{\frac{n}{2}}}
 (1-\chi_\text{in})\Lambda_y{\sf Q}(y)
 +
 {\sf N}(v).
 \end{align}
We can check that $v(x,t)$ defied in \eqref{5.3} gives a solution of \eqref{5.1},
if $(\epsilon(y,t),w(x,t))$ solves \eqref{5.5}.
By a lack of boundary condition in the equation for $\epsilon(y,t)$ in \eqref{5.5},
the problem may not be uniquely solvable.
So we appropriately construct a solution $\epsilon(y,t)$
such that $\epsilon(y,t)$ decays enough in the region $|y|\sim R$ (see \eqref{6.16}).

\subsection{Fixed point argument}
\label{Sec5.2}
To construct a solution of \eqref{5.5},
we apply a fixed point argument.
We put
 \[
 {\cal W}(x,t)
 =
 \begin{cases}
 (T-t)^l\left( 1+|z|^{2l+2} \right)
 & \text{for } |z|<(T-t)^{-\frac{l}{2l+2}},
 \\[1mm] \dis
 \frac{1}{1+|x|^2} & \text{for } |z|>(T-t)^{-\frac{l}{2l+2}},
 \end{cases}
 \qquad
 z=\frac{x}{\sqrt{T-t}}.
 \]
We fix two small positive constants $\delta_0$ and $\sigma$.
Let $X_\sigma$ be the space of  all continuous functions on $\R^5\times[0,T-2\sigma]$ satisfying
 \[
 |w(x,t)|
 \leq
 \delta_0{\cal W}(x,t)
 \quad
 \text{for }
 t\in[0,T-2\sigma].
 \]
The metric in $X_\sigma$ is defined by
 \begin{equation}\label{5.8}
 d_{X_\sigma}(w_1,w_2)=\|w_1-w_2\|_{C(\R^5\times[0,T-2\sigma])}.
 \end{equation}
We extend $w(x,t)$ to a continuous function on $\R^5\times[0,T]$.
 \[
 \bar{w}(x,t)=
 \begin{cases}
 w(x,t)
 & \text{if } (x,t)\in\R^5\times(0,T-2\sigma),
 \\
 \chi_\sigma(t)w(x,T-2\sigma)
 & \text{if } (x,t)\in\R^5\times(T-2\sigma,T),
 \end{cases}
 \]
where $\chi_\sigma(t)$ is a cut off function satisfying
$\chi_\sigma(t)=1$ if $t\in[0,T-\frac{3}{2}\sigma]$
and
$\chi_\sigma(t)=0$ if $t\in[T-\sigma,T]$.
Furthermore we define
 \[
 \tilde{w}(x,t)
 =
 \begin{cases}
 \bar{w}(x,t) & \text{if } t\in[0,T-2\sigma],
 \\
 \delta_0{\cal W}(x,t) & \text{if } t\in[T-2\sigma,T] \text{\ \ and\ \ }
 \bar{w}(x,t)>\delta_0{\cal W}(x,t),
 \\
 \bar{w}(x,t) & \text{if } t\in[T-2\sigma,T] \text{\ \ and\  \ }
 |\bar{w}(x,t)|\leq\delta_0{\cal W}(x,t),
 \\
 -\delta_0{\cal W}(x,t) & \text{if } t\in[T-2\sigma,T] \text{\ \ and\ \ }
 \bar{w}(x,t)<-\delta_0{\cal W}(x,t).
 \end{cases}
 \]
From this definition,
we see that $\tilde{w}(x,t)\in C(\R^5\times[0,T])$ and
 \begin{equation}\label{5.9}
 |\tilde{w}(x,t)|
 \leq
 \begin{cases}
 \delta_0{\cal W}(x,t) & \text{for } (x,t)\in\R^5\times(0,T-\sigma),
 \\
 0 & \text{for } (x,t)\in\R^5\times(T-\sigma,T).
 \end{cases}
 \end{equation}
For given $w(x,t)\in X_\sigma$,
we first determine $\lambda(t)$ by the orthogonal condition \eqref{6.1}.
Next
we construct $\epsilon(y,t)$ as a solution of
 \[
 \lambda^2\epsilon_t=H_y\epsilon+G_\text{in}(\lambda,\tilde w)
 \qquad\text{in } B_{2R}\times(0,T).
 \]
After that
we solve the problem
 \[
 W_t=\Delta_xW+G_\text{out}(\lambda,\tilde{w},\epsilon)
 \qquad\text{in } \R^5\times(0,T).
 \]
By using this $W(x,t)$,
we define the mapping $w(x,t)\mapsto W(x,t)$.
The fixed point of this mapping gives the desired solution of \eqref{5.5}.
Finally
we take $\sigma\to0$ to obtain the solution described in Theorem \ref{Thm1}.
For the rest of paper,
we construct the solution mapping in the above procedure.

\subsection{Notations}
\label{Sec5.3}
From now on,
we assume $R\gg1$ and choose $T$ as
 \begin{equation*}
 T=e^{-R}.
 \end{equation*}
For any positive constant $k_1$ and $k_2$,
we write
 \[
 k_1\lesssim k_2
 \]
if there is a universal constant $c>0$ independent of $R$, $\delta_0$, $\sigma$
such that $k_1\leq ck_2$.

\section{Inner solution}
\label{Sec6}
In this section,
we repeat the argument in Lemma 4.1 \cite{delPino3} to define the mapping
$w(x,t)\in X_\sigma\mapsto\epsilon(y,t)$ mentioned in Section \ref{Sec5.2}.
Throughout this section,
$\tilde{w}(x,t)\in C(\R^5\times[0,T])$
represents an extension of $w(x,t)\in X_\sigma$ defined in Section \ref{Sec5.2}.

\subsection{Choice of $\lambda(t)$}
We define $\lambda(t)\in C^1([0,T])$ as the unique solution of
 \begin{equation}\label{6.1}
 (\chi_{4R}G_\text{in}(t),\Lambda_y{\sf Q})_{L_y^2(B_{8R})}=0
 \quad\text{for } t\in(0,T)
 \qquad
 \text{and}
 \qquad
 \lambda(T)=0,
 \end{equation}
where $G_\text{in}(t)=G_\text{in}(\lambda(t),\tilde w(\lambda(t)y,t))$
and $\chi_{4R}(y)=\chi(\frac{|y|}{4R})$.
 \begin{lem}\label{Lem6.1}
 There exists $K>1$ independent of $R$, $\delta_0$, $\sigma$ such that
 \begin{align*}
 \left| \lambda(t)-\alpha_l^2(T-t)^{2l+2} \right|
 &<
 K\left( \delta_0+\frac{T-t}{R^2} \right)(T-t)^{2l+2}
 \qquad \text{\rm for}\ t\in(0,T),
 \\
 \left| \frac{d\lambda}{dt}(t)-(2l+2)\alpha_l^2(T-t)^{2l+1} \right|
 &<
 K\left( \delta_0+\frac{T-t}{R^2} \right)(T-t)^{2l+1}
 \qquad \text{\rm for}\ t\in(0,T),
 \end{align*}
 where
 $\alpha_l=-\frac{(\chi_{4R}V,\Lambda_y{\sf Q})_{L_y^2(B_{8R})}}
 {2(l+1)(\chi_{4R}\Lambda_y{\sf Q},\Lambda_y{\sf Q})_{L_y^2(B_{8R})}}$.
 \end{lem}
\begin{proof}
From \eqref{5.6}, the orthogonal condition \eqref{6.1} is explicitly given by
 \[
 \sqrt{\lambda}(\chi_{4R}V\Theta(t),\Lambda_y{\sf Q})_{L_y^2(B_{8R})}
 +
 \sqrt{\lambda}(\chi_{4R}V\tilde w(t),\Lambda_y{\sf Q})_{L_y^2(B_{8R})}
 +
 \frac{d\lambda}{dt}(\chi_{4R}\Lambda_y{\sf Q},\Lambda_y{\sf Q})_{L_y^2(B_{8R})}
 =
 0.
 \]
Put $C(t)=\sqrt{\lambda(t)}$.
The function $C(t)$ satisfies
 \begin{equation}\label{6.2}
 \frac{dC}{dt}
 =
 -\frac{(\chi_{4R}V\Theta(x,t),\Lambda_y{\sf Q})_{L_y^2(B_{8R})}
 +
 (\chi_{4R}V\tilde w(x,t),\Lambda_y{\sf Q})_{L_y^2(B_{8R})}}
 {2(\chi_{4R}\Lambda_y{\sf Q},\Lambda_y{\sf Q})_{L_y^2(B_{8R})}},
 \qquad
 x=C^2y.
 \end{equation}
We now show the unique solvability of the problem \eqref{6.2} in
 \[
 {\cal S}=
 \left\{
 C(t)\in C([0,T]);\ 0\leq C(t)\leq\frac{\sqrt{T-t}}{R}
 \right\}.
 \]
Let us consider
 \[
 \frac{dD}{dt}
 =
 -\frac{(\chi_{4R}V\Theta(x,t),\Lambda_y{\sf Q})_{L_y^2(B_{8R})}
 +
 (\chi_{4R}V\tilde w(x,t),\Lambda_y{\sf Q})_{L_y^2(B_{8R})}}
 {2(\chi_{4R}\Lambda_y{\sf Q},\Lambda_y{\sf Q})_{L_y^2(B_{8R})}},
 \qquad x=C^2y.
 \]
This differential equation defines the mapping $C(t) \mapsto D(t)$.
Put $z=\frac{x}{\sqrt{T-t}}$.
Since $e_l(z)=1+a_1|z|^2+a_2|z|^4+\cdots+a_l|z|^{2l}$,
the first term on the right-hand side is written as
 \begin{align*}
 (\chi_{4R}V\Theta(x,t),&
 \Lambda_y{\sf Q})_{L_y^2(B_{8R})}
 =
 -(T-t)^l(Ve_l(z),\Lambda_y{\sf Q})_{L_y^2(B_{8R})}
 \\
 &=
 -(T-t)^l
 \left(
 \left(
 \chi_{4R}V,\Lambda_y{\sf Q}
 \right)_{L_y^2(B_{8R})}
 +
 \sum_{k=1}^l
 a_k
 \left(
 \chi_{4R}V|z|^{2k},\Lambda_y{\sf Q}
 \right)_{L_y^2(B_{8R})}
 \right)
 \\
 &=
 -(T-t)^l
 \left(
 \underbrace{
 \left(
 \chi_{4R}V,\Lambda_y{\sf Q}
 \right)_{L_y^2(B_{8R})}}_{=:b_0}
 +
 \sum_{k=1}^l
 \underbrace{
 a_k
 \left(
 \chi_{4R}V\left( \frac{|y|}{R} \right)^{2k},\Lambda_y{\sf Q}
 \right)_{L_y^2(B_{8R})}}_{=:b_k}
 \frac{R^{2k}C(t)^{4k}}{(T-t)^k}
 \right).
 \end{align*}
Since $|b_k|\lesssim1$,
we easily see that if $C(t)\in{\cal S}$
 \[
 \left|
 \sum_{k=1}^lb_k\frac{R^{2k}C(t)^{4k}}{(T-t)^k}
 \right|
 \lesssim
 \frac{T-t}{R^2}
 \qquad\text{for } t\in(0,T).
 \]
Furthermore
when $|z|=\frac{C(t)^2}{\sqrt{T-t}}|y|$ and $C(t)\in{\cal S}$,
it holds that $|z|<1$ for $y\in B_{8R}$.
Therefore
since $\tilde{w}(x,t)$ satisfies \eqref{5.9},
it follows that if $C(t)\in{\cal S}$
 \begin{align*}
 |(\chi_{4R}V\tilde w(C^2y,t),\Lambda_y{\sf Q})_{L_y^2(B_{8R})}|
 &\lesssim
 \delta_0
 (T-t)^l
 |(\chi_{4R}V,\Lambda_y{\sf Q})_{L_y^2(B_{8R})}|
 \lesssim
 \delta_0
 (T-t)^l.
 \end{align*}
Therefore
since $b_0<0$ (see \eqref{4.3}),
we get if $C(t)\in{\cal S}$
 \[
 \frac{dD}{dt}
 =
 -
 \left( 1+O(\delta_0)+O\left( \frac{T-t}{R^2} \right) \right)
 \underbrace{
 \left(
 \frac{|b_0|}
 {2(\chi_{4R}\Lambda_y{\sf Q},\Lambda_y{\sf Q})_{L_y^2(B_{8R})}}
 \right)}_{=:(l+1)\alpha_l}
 (T-t)^l.
 \]
This implies
 \begin{equation}\label{6.3}
 D(t)
 =
 \left( 1+O(\delta_0)+O\left( \frac{T-t}{R^2} \right) \right)
 \alpha_l(T-t)^{l+1}.
 \end{equation}
Therefore
we proved that $D(t)\in{\cal S}$ if $C(t)\in{\cal S}$.
As a consequence,
by a fixed point argument,
we obtain a solution $C(t)$ of \eqref{6.2} satisfying \eqref{6.3}.
Next
we prove the uniqueness for solutions of \eqref{6.2} in ${\cal S}$.
Let $C_1(t),C_2(t)\in{\cal S}$ be two solutions of \eqref{6.2}.
Since $\tilde{w}(x,t)=0$ for $t\in(T-\sigma,T)$,
it is clear that $C_1(t)=C_2(t)$ for $t\in(T-\sigma,T)$.
We write $x_1=C_1^2y$, $x_2=C_2^2y$ and $y'=\frac{C_1^2}{C_2^2}y$.
By the change of variables,
we see that
 \begin{align*}
 &
 \left| \frac{d}{dt}(C_1-C_2) \right|
 \lesssim
 \left|
 \int_{B_{8R}}
 \chi_{4R}V\tilde w(x_1,t)\Lambda_y{\sf Q}
 dy
 -
 \int_{B_{8R}}
 \chi_{4R}V\tilde w(x_2,t)\Lambda_y{\sf Q}
 dy
 \right|
 \\
 &\qquad=
 \left|
 \int_{B_{8R}}
 \chi_{4R}V\tilde{w}(x_1,t)\Lambda_y{\sf Q}
 dy
 -
 \left( \frac{C_1}{C_2} \right)^{10}
 \int_{B_{8(\frac{C_2}{C_1})^2R}}
 \chi_{4R}(y')V(y')\tilde{w}(x_1,t)\Lambda_y{\sf Q}(y')
 dy
 \right|
 \\
 &\qquad\lesssim
 \left|
 1-\left( \frac{C_1}{C_2} \right)^{10}
 \right|
 \int_{B_{8R}}
 \left| \chi_{4R}V\tilde w(x_1,t)\Lambda_y{\sf Q} \right|
 dy
 +
 \left( \frac{C_1}{C_2} \right)^{10}
 \left|
 \int_{B_{8R}}
 \chi_{4R}V\tilde w (x_1,t)\Lambda_y{\sf Q}
 dy
 \right.
 \\
 &\qquad\qquad
 \left.
 -
 \int_{B_{8(\frac{C_2}{C_1})^2R}}
 \chi_{4R}(y')V(y')\tilde{w}(x,t)\Lambda_y{\sf Q}(y')
 dy
 \right|.
 \end{align*}
Repeating the above argument,
we can verify that $C_1(t),C_2(t)$ satisfy \eqref{6.3}.
Therefore
we find that  $C_1(t),C_2(t)>\frac{\alpha_l}{2}\sigma^{l+1}$ for $t\in(0,T-\sigma)$.
As a consequence,
there exists $c_\sigma>0$ such that
 \begin{align*}
 \left| \frac{d}{dt}(C_1-C_2) \right|
 <
 c_\sigma|C_1-C_2|
 \qquad\text{for } t\in(0,T-\sigma).
 \end{align*}
This assures the uniqueness of solutions in $t\in(0,T-\sigma)$.
Therefore
the uniqueness of \eqref{6.2} in ${\cal S}$ is proved.
Since $C(t)=\sqrt{\lambda(t)}$,
we obtain the conclusion.
\end{proof}

\subsection{Construction of $\epsilon(y,t)$}
\label{Sec6.2}
Throughout this subsection,
$\lambda(t)$ represents the function given in Lemma \ref{Lem6.1}.
For simplicity,
we write
 \[
 G_\text{in}(t)=G_\text{in}(\lambda(t),\tilde w(\lambda(t)y,t)).
 \]
We define a function $g(y,t)\in C(\R^5\times[0,T])$ as a solution of
 \[
 \begin{cases}
 -H_yg
 =
 \chi_{4R}G_\text{in}(t)
 & \text{in } B_{8R},
 \\
 g=0
 & \text{on } \pa B_{8R},
 \\
 g(\cdot,t) \text{ is radial}.
 \end{cases}
  \]
 Since
 $(\chi_{4R}G_\text{in}(t),\Lambda_y{\sf Q})_{L_y^2(B_{8R})}=0$
 for $t\in(0,T)$ (see \eqref{6.1}),
 the radial solution $g(r,t)$ is given by
 \[
 g(r,t)
 =
 k\Gamma(r)\int_r^{8R}
 \Lambda_y{\sf Q}
 \chi_{4R}G_\text{in}(t)r'^4dr'
 -
 k\Lambda_y{\sf Q}(r)\int_r^{8R}
 \Gamma\chi_{4R}G_\text{in}(t)r'^4dr'
 \]
for some constant $k\in\R$ depending on $\Lambda_y{\sf Q}(r)$ and $\Gamma(r)$.
The function $\Gamma(r)$ is a radial solution of $H_y\psi=0$
given in the proof of Lemma \ref{Lem3.3}.
From \eqref{5.6} and Lemma \ref{Lem6.1},
we verify that
 \begin{align*}
 |G_\text{in}|
 &=
 \left|
 \lambda^\frac{3}{2}\Theta(x,t)V
 +
 \lambda^\frac{3}{2}\tilde w(x,t)V
 +
 \lambda\lambda_t\Lambda_y{\sf Q}
 \right|
 \\
 &\lesssim
 \frac{\lambda^\frac{3}{2}(T-t)^l}{1+|y|^4}
 +
 \frac{\lambda|\lambda_t|}{1+|y|^3}
 <
 \frac{\lambda^\frac{3}{2}\lambda^\frac{l}{2l+2}}{1+|y|^4}
 +
 \frac{\lambda\lambda^{\frac{2l+1}{2l+2}}}{1+|y|^3}
 \\
 &\lesssim
 \frac{\lambda^{\frac{4l+3}{2l+2}}}{1+|y|^3}
 \qquad\text{for } |y|<8R.
 \end{align*}
Therefore
by a direct computation,
we get
 \begin{equation}\label{6.4}
 |g(y,t)|
 <
 \frac{\lambda^\frac{4l+3}{2l+2}}{1+|y|}
 \qquad\text{for } |y|<8R.
 \end{equation}
For simplicity,
we put
 \[
 \gamma=\frac{4l+3}{2l+2}.
 \]
We introduce a new time variable $s$ defined by
 \begin{equation}\label{6.5}
 \frac{ds}{dt}=\frac{1}{\lambda(t)^2} \qquad\text{and}\qquad s(t)|_{t=0}=0.
 \end{equation}
Since $\frac{\alpha_l^2}{2}(T-t)^{2l+2}<\lambda<2\alpha_l^2(T-t)^{2l+2}$
(see Lemma \ref{Lem6.1}),
it is expressed in the variable $s$.
 \begin{equation}\label{6.6}
 \frac{\alpha_l^2}{2}
 \left(
 \frac{1}{T^{-(4l+3)}+\frac{4l+3}{4}\alpha_l^2s}
 \right)^{\frac{2l+2}{4l+3}}
 <
 \lambda
 <
 2\alpha_l^2
 \left(
 \frac{1}{T^{-(4l+3)}+4(4l+3)\alpha_l^2s}
 \right)^{\frac{2l+2}{4l+3}}.
 \end{equation}
Let $\mu_1^{(8R)}<0$ and $\psi_1^{(8R)}(y)\in H_0^1(B_{8R})$ be defined
in Section \ref{Sec3.2}.
We consider
 \[
 \begin{cases}
 \pa_s E
 =
 H_yE+g(y,s)
 &
 \text{in } B_{8R}\times(0,\infty),
 \\
 E=0
 &
 \text{on } \pa B_{8R}\times(0,\infty),
 \\
 \dis
 E=\frac{{\sf d_\text{in}}}{\mu_1^{(8R)}}\psi_1^{(8R)}
 &
 \text{for } s=0.
 \end{cases}
 \]
The parameter ${\sf d_\text{in}}$ is determined below.
The desired solution $\epsilon(y,t)$ mentioned in section \ref{Sec5.2}
is obtained by $\epsilon(y,t)=H_yE(y,t)$.
Let $M_1$ be the constant given in Lemma \ref{Lem3.4} and fix a large constant $M>M_1$
such that 
 \begin{equation}\label{6.7}
 |y|^2V(y)+|y|^\frac{5}{2}e^{-\sqrt{|\mu_1|}|y|}\ll1
 \qquad\text{for } |y|>2M.
 \end{equation}
We first consider
 \begin{equation}\label{6.8}
 \begin{cases}
 \pa_s E_1
 =
 \Delta_yE_1+(1-\chi_M)V(y)E_1+g(y,s)
 &
 \text{in } B_{8R}\times(0,\infty),
 \\
 E_1=0
 &
 \text{on } \pa B_{8R}\times(0,\infty),
 \\
 E_1=0
 &
 \text{for } s=0.
 \end{cases}
 \end{equation}
Since $g(y,t)=0$ near $\pa B_{8R}\times[0,\infty)$,
by a certain approximation procedure,
we can verify that there exists $\alpha\in(0,1)$ such that
 \begin{equation}\label{6.9}
 E_1,\
 \nabla_yE_1,\
 \Delta_yE_1,\
 \nabla_y\Delta_yE_1
 \in C^{\alpha}(\bar{B}_{8R}\times[0,\infty)).
 \end{equation}
\begin{lem}\label{Lem6.2}
 There exists $K_1>1$ independent of $R$, $\delta_0$, $\sigma$ such that
 \[
  |E_1(y,s)|+|\nabla E_1(y,s)|\leq K_1R\lambda^\gamma
 \qquad \text{\rm for}\ (y,s)\in B_{8R}\times(0,\infty).
 \]
\end{lem}
\begin{proof}
To construct a comparison function,
we put
 \[
 p(r)
 =
 p_1(r)\int_r^{16R}\frac{dr_1}{p_1(r_1)^2r_1^{n-1}}\int_0^{r_1}
 \frac{p_1(r_2)r_2^{n-1}}{1+r_2}dr_2,
 \]
where $p_1(r)$ is a radial function given in Lemma \ref{Lem3.4}.
The function $p(r)$ gives a positive radial solution of
 \[
 \begin{cases}
 \dis
 \Delta_yp+(1-\chi_M)Vp+\frac{1}{1+|y|}=0 & \text{in } B_{16R},
 \\
 p=0 & \text{on } \pa B_{16R}.
 \end{cases}
 \]
Since $k<p_1(r)<1$ for $r>0$ (see Lemma \ref{Lem3.4}),
there exist $k_1$, $k_2>0$ independent of $M$, $R$ such that
 \begin{equation}\label{6.10}
 k_1R<p(r)<k_2R
 \qquad\text{for } r\in(0,8R).
 \end{equation}
We now check that $K\lambda^\gamma p(y)$ gives a super-solution of \eqref{6.8}.
Since
$|\lambda\frac{d\lambda}{dt}|<(4l+4)\alpha_l^4(T-t)^{4l+3}
 <(4l+4)\alpha_l^4T^{4l+3}$
(see Lemma \ref{Lem6.1}),
we see from \eqref{6.5} and \eqref{6.10} that
 \begin{align*}
 (\pa_s-\Delta_y-(1-\chi_M)V)\lambda^\gamma p(y)
 &=
 \left(
 \frac{\gamma}{\lambda}\frac{dt}{ds}\frac{d\lambda}{dt}p(y)
 +
 \frac{1}{1+|y|}
 \right)
 \lambda^\gamma
 \\
 &>
 \left(
 -\gamma\lambda\frac{d\lambda}{dt}p(y)+\frac{1}{1+|y|}
 \right)
 \lambda^\gamma
 \\
 &>
 \left(
 -(4l+4)\gamma\alpha_l^4k_2T^{4l+3}R+\frac{1}{1+|y|}
 \right)
 \lambda^\gamma.
 \end{align*}
Since $T=e^{-R}$,
it holds that
 \[
 (4l+4)\gamma\alpha_l^4k_2T^{4l+3}R
 <
 \frac{\frac{1}{2}}{1+8R}
 < 
 \frac{\frac{1}{2}}{1+|y|}
 \qquad\text{for } y\in B_{8R}.
 \]
Therefore
we obtain
 \begin{align*}
 (\pa_s-\Delta_y-(1-\chi_M)V)K\lambda^\gamma p(y)
 >
 \frac{K}{2}\frac{\lambda^\gamma}{1+|y|}
 \qquad
 \text{for } y\in B_{8R}.
 \end{align*}
Since $|g(y,s)|\lesssim\frac{\lambda^\gamma}{1+|y|}$ (see \eqref{6.4}),
by a comparison argument,
we obtain if $K\gg1$
 \[
 |E_1(y,s)|
 <
 K\lambda^\gamma p(y)
 <
 Kk_2R\lambda^\gamma
 \qquad \text{for } (y,s)\in B_{8R}\times(0,\infty).
 \]
Applying a local parabolic estimate in \eqref{6.8},
we get from \eqref{6.4} that
 \begin{align*}
 |\nabla_yE_1(y,s)|
 &\lesssim
 \sup_{\min\{s-1,0\}<s'<s}\sup_{|y'-y|<1}
 \left( |E_1(y',s')|+|g(y',s')| \right)
 \nonumber
 \\
 &\lesssim
 \sup_{\min\{s-1,0\}<s'<s}
 \left( R\lambda(s')^\gamma+\frac{\lambda(s')^\gamma}{1+|y|} \right).
 \end{align*}
From \eqref{6.6},
we can verify that
 \begin{equation}\label{6.11}
 \sup_{\min\{s-1,0\}<s'<s}\lambda(s')\lesssim\lambda(s).
 \end{equation}
Therefore we complete the proof.
\end{proof}
 \begin{lem}\label{Lem6.3}
 There exists $K_2>1$ independent of $R$, $\delta_0$, $\sigma$ such that
 \[
 |\Delta_yE_1(y,s)|+|y|\cdot|\nabla_y\Delta_yE_1(y,s)|
 <
 K_2
 \left(
 \frac{\lambda^\gamma}{1+|y|}
 +
 \frac{R\lambda^\gamma}{1+|y|^2}
 \right)
 \qquad\text{\rm for}\ (y,s)\in B_{6R}\times(0,\infty).
 \]
\end{lem}
\begin{proof}
We put
 \[
 e_1
 =
 \Delta_yE_1+(1-\chi_M)V(y)E_1.
 \]
We easily see from \eqref{6.8} - \eqref{6.9} that $e_1(y,s)$ solves
 \[
 \begin{cases}
 \pa_se_1
 =
 \Delta_ye_1+(1-\chi_M)V(y)e_1+(\Delta_y+(1-\chi_M)V(y))g(y,s)
 &
 \text{in } B_{8R}\times(0,\infty),
 \\
 e_1=0 & \text{on } \pa B_{8R}\times(0,\infty),
 \\
 e_1=0 & \text{for } s=0.
 \end{cases}
 \]
Since $|(\Delta_y+(1-\chi_M)V(y))g(y,s)|<\frac{\lambda^\gamma}{1+|y|^3}$,
by the same argument as in the proof of Lemma \ref{Lem6.2},
we verify that
 \begin{equation}\label{6.12}
 e_1(y,s)
 \lesssim
 \frac{\lambda^\gamma}{1+|y|}
 \qquad\text{for } (y,s)\in B_{8R}\times(0,\infty).
 \end{equation}
To extend $e_1(y,s)$ to $s<0$,
we put
 \[
 \bar{e}_1(y,s)=
 \begin{cases}
 e_1(y,s) & \text{if } s\geq0,
 \\
 0 & \text{if } s<0,
 \end{cases}
 \hspace{15mm}
 \bar{g}(y,s)=
 \begin{cases}
 g(y,s) & \text{if } s\geq0,
 \\
 0 & \text{if } s<0.
 \end{cases}
 \]
Since $e_1,\nabla_ye_1\in C^{\alpha}(\bar B_{8R}\times[0,\infty))$ (see \eqref{6.9}),
we find that
$\bar{e}_1,\nabla_x\bar{e}_1\in C^{\alpha}(\bar B_{8R}\times(-\infty,\infty))$ and
${\bar e}_1(y,s)$ solves
 \[
 \begin{cases}
 \pa_s\bar e_1
 =
 \Delta_y\bar e_1
 +
 (1-\chi_M)V(y)\bar e_1
 +
 (\Delta_y+(1-\chi_M)V(y))\bar g(y,s)
 &
 \text{in } B_{8R}\times(-\infty,\infty),
 \\
 \bar e_1=0 & \text{on } \pa B_{8R}\times(-\infty,\infty).
 \end{cases}
 \]
We fix $(y,s)\in B_{6R}\setminus B_1\times(0,\infty)$.
We write $\rho=|y|$ and define
 \[
 \tilde e_1(Y,S)
 =
 \bar e_1\left( y+\frac{\rho}{3}Y,s+\frac{\rho^2}{9}(S-1) \right).
 \]
We easily see that $\tilde e_1(Y,S)$ solves
 \[
 \dis
 \pa_S\tilde e_1
 =
 \Delta_Y\tilde e_1
 +
 \frac{\rho^2}{9}(1-\chi_M)V\tilde e_1
 +
 \frac{\rho^2}{9}(\Delta_y+2(1-\chi_M)V)\bar g
 \qquad
 \text{in } B_1\times(0,1).
 \]
Since $|y|=\rho$,
it holds that
 \[
 \frac{\rho^2}{9}(1-\chi_M)V
 \lesssim
 \frac{\rho^2}{1+|y+\frac{\rho}{2}Y|^4}
 \lesssim
 \frac{\rho^2}{1+\rho^4}
 \qquad\text{for } |Y|<1.
 \]
Therefore
we get from Lemma \ref{Lem3.5} that
 \begin{align*}
 \|\tilde e_1\|_{C_{(Y,S)}^{\alpha}(B_\frac{1}{2}\times(\frac{1}{2},1))}
 &+
 \|\nabla_Y\tilde e_1\|_{C_{(Y,S)}^{\alpha}(B_\frac{1}{2}\times(\frac{1}{2},1))}
 \\
 &\lesssim
 \|\tilde e_1\|_{L_{(Y,S)}^\infty(B_1\times(0,1))}
 +
 \left\|
 \frac{\rho^2}{9}(\Delta_y+2(1-\chi_M)V)\bar g
 \right\|_{L_{(Y,S)}^\infty(B_1\times(0,1))}.
 \end{align*}
Since $|(\Delta_y+(1-\chi_M)V(y))g(y,s)|<\frac{\lambda^\gamma}{1+|y|^3}$,
we see that
 \[
 \left\|
 \frac{\rho^2}{9}(\Delta_y+2(1-\chi_M)V)\bar g
 \right\|_{L_{(Y,S)}^\infty(B_1\times(0,1))}
 \lesssim
 \frac{\rho^2}{1+\rho^3}
 \left( \sup_{\min\{0,s-\frac{\rho^2}{9}\}<s'<s}\lambda(s') \right)^\gamma.
 \]
From \eqref{6.6} and $T=e^{-R}$,
we verify that
 \[
 \sup_{\min\{0,s-\frac{\rho^2}{9}\}<s'<s}\lambda(s')
 \lesssim
 \lambda(s).
 \]
Since
$\nabla_Y\tilde e_1(Y,S)
 =\frac{\rho}{3}\nabla_y\bar e_1(y+\frac{\rho}{3}Y,s+\frac{\rho^2}{9}S)$,
we deduce from \eqref{6.12} that
 \begin{align*}
 \frac{\rho}{3}|\nabla_ye_1(y,s)|
 &<
 \|\nabla_Y\tilde e_1\|_{C_{(Y,S)}^\alpha(B_\frac{1}{2}\times(\frac{1}{2},1))}
 \\
 &\lesssim
 \|\tilde e_1\|_{L_{(Y,S)}^\infty(B_1\times(0,1))}
 +
 \frac{\rho^2}{1+\rho^3}
 \lambda(s)^\gamma
 \\
 &\lesssim
 \frac{\lambda^\gamma}{1+\rho}
 \end{align*}
Therefore
it follows that
 \[
 |\nabla_ye_1(y,s)|
 \lesssim
 \frac{\lambda^\gamma}{1+|y|^2}
 \qquad
 \text{for } (y,s)\in B_{6R}\setminus B_1\times(0,\infty).
 \]
Combining this estimate and \eqref{6.12},
we obtain
 \[
 |\nabla_ye_1(y,s)|
 \lesssim
 \frac{\lambda^\gamma}{1+|y|^2}
 \qquad
 \text{for } (y,s)\in B_{6R}\times(0,\infty).
 \]
From definition of $e_1(y,s)$ and Lemma \ref{Lem6.2},
we complete the proof.
\end{proof}
Next we put
 \[
 E_2=E-E_1.
 \]
The function $E_2(y,s)$ solves
 \[
 \begin{cases}
 \pa_s E_2
 =
 H_yE_2+\chi_MV(y)E_1
 &
 \text{in } B_{8R}\times(0,\infty),
 \\
 E_2=0
 &
 \text{on } \pa B_{8R}\times(0,\infty),
 \\
 \dis
 E_2=\frac{{\sf d_\text{in}}}{\mu_1^{(8R)}}\psi_1^{(8R)}
 &
 \text{for } s=0.
 \end{cases}
 \]
We now take
 \[
 {\sf d_\text{in}}=-\mu_1^{(8R)}\int_{0}^\infty
 e^{\mu_1^{(8R)}s'}
 \left( \chi_MVE_1(s'),\psi_1^{(8R)} \right)_{L_y^2(B_{8R})}ds'
 \]
and define ${\sf c}(s)$ by
 \[
 \begin{cases}
 \dis
 \frac{d{\sf c}}{ds}
 =
 -\mu_1^{(8R)}{\sf c}+\left( \chi_MVE_1(s),\psi_1^{(8R)} \right)_{L_y^2(B_{8R})},
 \\
 \dis
 {\sf c}(0)=\frac{{\sf d_\text{in}}}{\mu_1^{(8R)}}.
 \end{cases}
 \]
The function ${\sf c}(s)$ is explicitly given by
 \[
 {\sf c}(s)=-\int_s^\infty e^{-\mu_1^{(8R)}(s-s')}
 \left( \chi_MVE_1(s'),\psi_1^{(8R)} \right)_{L_y^2(B_{8R})}
 ds'.
 \]
From \eqref{6.6} and Lemma \ref{Lem6.2},
we easily see that
 \begin{equation}\label{6.13}
 |{\sf c}(s)|\lesssim R\lambda^\gamma,
 \hspace{10mm}
 |{\sf d_\text{in}}|\lesssim R\lambda(s)|_{s=0}^\gamma\lesssim RT^{4l+3}.
 \end{equation}
We decompose $E_2(y,s)$ as
 \[
 E_2=\nu+{\sf c}(s)\psi_1^{(8R)}.
 \]
The function $\nu(y,s)$ satisfies
 \begin{equation}\label{6.14}
 \begin{cases}
 \pa_s\nu
 =
 H_y\nu
 +
 V\chi_ME_1
 -
 \left( \chi_MVE_1,\psi_1^{(8R)} \right)_{L_y^2(B_{8R})}
 \psi_1^{(8R)}
 &
 \text{in } B_{8R}\times(0,\infty),
 \\
 \nu=0
 &
 \text{on } \pa B_{8R}\times(0,\infty),
 \\
 \nu=0
 &
 \text{for } s=0.
 \end{cases}
 \end{equation}
 \begin{lem}\label{Lem6.4}
 There exists $K_3>1$ independent of $R$, $\delta_0$, $\sigma$ such that
 \begin{align*}
 |\nu(y,s)|
 +
 |y|\cdot|\nabla_y\nu(y,s)|
 +
 |y|^2\cdot|\Delta_y\nu(y,s)|
 +
 |y|^3\cdot|\nabla_y\Delta_y\nu(y,s)|
 <
 \frac{K_3R^4\lambda^\gamma}{1+|y|^\frac{5}{2}}
 \end{align*}
for $(y,s)\in B_{2R}\times(0,\infty)$.
 \end{lem}
\begin{proof}
Since $(\nu(s),\psi_1^{(8R)})_{L_y^2(B_{8R})}=0$ for $s\in(0,\infty)$,
from Lemma \ref{Lem3.3},
there exists $k>0$ such that
 \[
 (H_y\nu(s),\nu(s))_{L_y^2(B_{8R})}
 <
 -\frac{k}{R^3}\|\nu(s)\|_{L_y^2(B_{8R})}
 \qquad\text{for } s\in(0,\infty).
 \]
From this estimate and Lemma \ref{Lem6.2},
we get
 \begin{align*}
 e^{\frac{k}{R^3}s}\|\nu\|_{L_y^2(B_{8R})}^2
 &\lesssim
 R^5\|\chi_MV\|_{L_y^2(B_{8R})}^2
 \int_0^se^{\frac{k}{R^3}s'}
 \lambda(s')^{2\gamma}ds'.
 \end{align*}
We calculate the integral using Lemma \ref{Lem6.1} and \eqref{6.5}.
 \begin{align*}
 \int_0^s
 e^{\frac{k}{R^3}s'}
 \lambda(s')^{2\gamma}ds'
 &<
 R^3
 e^{\frac{k}{R^3}s}
 \lambda(s)^{2\gamma}
 -
 2\gamma R^3
 \int_0^s
 e^{\frac{k}{R^3}s'}
 \lambda^{2\gamma-1}\frac{d\lambda}{ds}ds'
 \\
 &=
 R^3 e^{\frac{k}{R^3}s}
 \lambda(s)^{2\gamma}
 -
 2\gamma R^3
 \int_0^s
 e^{\frac{k}{R^3}s'}
 \lambda^{2\gamma+1}\frac{d\lambda}{dt}ds'
 \\
 &\lesssim
 R^3e^{\frac{k}{R^3}s}
 \lambda(s)^{2\gamma}
 +
 R^3
 \int_0^s
 e^{\frac{k}{R^3}s'}
 \lambda^{2\gamma+1}\lambda^\frac{2l+1}{2l+2}ds'
 \\
 &\lesssim
 R^3e^{\frac{k}{R^3}s}
 \lambda(s)^{2\gamma}
 +
 R^3\lambda(s)|_{s=0}^\frac{4l+3}{2l+2}
 \int_0^s
 e^{\frac{k}{R^3}s'}
 \lambda^{2\gamma}ds'.
 \end{align*}
Since $\lambda(s)|_{s=0}=\lambda(t)|_{t=0}\lesssim T^{2l+2}$ and $T=e^{-R}$,
it follows that
$R^3\lambda(s)|_{s=0}^\frac{4l+3}{2l+2}<\frac{1}{2}$.
Therefore we obtain
 \begin{align*}
 \int_0^s
 e^{\frac{k}{R^3}s'}
 \lambda(s')^{2\gamma}ds'
 \lesssim
 R^3e^{\frac{k}{R^3}s}\lambda(s)^{2\gamma}.
 \end{align*}
As a consequence,
we deduce that
 \[
 \|\nu(s)\|_{L_y^2(B_{8R})}\lesssim R^4\lambda^\gamma.
 \]
Applying a local parabolic estimate in \eqref{6.14},
we get from \eqref{6.11} that
 \begin{align}\label{6.15}
 \|\nu(s)\|_{L_y^\infty(B_{8R})}
 \lesssim
 R^4\lambda^\gamma.
 \end{align}
We now check that $KR^4\lambda^\gamma|y|^{-\frac{5}{2}}$
becomes a super solution for $y\in B_{8R}\setminus B_{2M}$.
From \eqref{6.7}, Lemma \ref{Lem6.1} and \eqref{6.5},
we see that
 \begin{align*}
 (\pa_s-H_y)
 \left(
 \frac{\lambda^\gamma}{|y|^{\frac{5}{2}}}
 \right)
 &=
 \left(
 \frac{\gamma}{\lambda}\frac{d\lambda}{dt}\frac{dt}{ds}
 +
 \frac{5}{4|y|^2}
 -
 V(y)
 \right)
 \left( \frac{\lambda^\gamma}{|y|^{\frac{5}{2}}} \right)
 \nonumber
 \\
 &>
 \left(
 \gamma\lambda\frac{d\lambda}{dt}
 +
 \frac{5}{8|y|^2}
 \right)
 \left( \frac{\lambda^\gamma}{|y|^{\frac{5}{2}}} \right)
 \nonumber
 \\
 &>
 \left(
 -(4l+4)\gamma\alpha_l^4T^{4l+3}+\frac{5}{8|y|^2}
 \right)
 \left( \frac{\lambda^\gamma}{|y|^{\frac{2}{2}}} \right)
 \nonumber
 \end{align*}
for $y\in B_{8R}\setminus B_{2M}$.
Since $T=e^{-R}$,
we note that $(4l+4)\gamma\alpha_l^4T^{4l+3}<\frac{5}{16|y|^2}$ for $y\in B_{8R}$.
Therefore
we get
 \begin{align*}
 (\pa_s-H_y)
 \left(
 \frac{KR^4\lambda^\gamma}{|y|^{\frac{5}{2}}}
 \right)
 >
 KR^4
 \left(
 \frac{5}{16|y|^2}
 \right)
 \left( \frac{\lambda^\gamma}{|y|^\frac{5}{2}} \right)
 \qquad\text{for } y\in B_{8R}\setminus B_{2M}.
 \end{align*}
Furthermore
from Lemma \ref{Lem3.1}, \eqref{6.7} and Lemma \ref{Lem6.2},
we see that
 \[
 \left( V\chi_ME_1,\psi_1^{(8R)} \right)_{L_y^2(B_{8R})}
 \psi_1^{(8R)}(y)
 \lesssim
 R\lambda^\gamma\frac{e^{-\sqrt{|\mu_1|}\cdot|y|}}{|y|^2}
 \lesssim
 \frac{R\lambda^\gamma}{|y|^{2+\frac{5}{2}}}
 \qquad \text{for } y\in B_{8R}\setminus B_{2M}.
 \]
Therefore
it holds that if $K\gg1$
 \begin{align*}
 (\pa_s-H_y)
 \left(
 \frac{KR^4\lambda^\gamma}{|y|^{\frac{5}{2}}}
 \right)
 >
 \left( V\chi_ME_1,\psi_1^{(8R)} \right)_{L_y^2(B_{8R})}
 \psi_1^{(8R)}
 \qquad \text{for } y\in B_{8R}\setminus B_{2M}.
 \end{align*}
Combining this estimate and \eqref{6.15},
by a comparison argument in \eqref{6.14},
we obtain
 \[
 |\nu(y,s)|
 <
 \frac{KR^4\lambda^\gamma}{|y|^{\frac{5}{2}}}
 \qquad \text{for } (y,s)\in B_{8R}\setminus B_{2M}\times(0,\infty).
 \]
By the same scaling argument as in the proof of Lemma \ref{Lem6.3},
we get
 \[
 |\nabla_y\nu(y,s)|
 \lesssim
 \frac{R^4\lambda^\gamma}{|y|^{\frac{7}{2}}}
 \qquad \text{for } (y,s)\in B_{6R}\setminus B_{2M}\times(0,\infty).
 \]
Next we consider the equation for $\pa_{y_i}\nu(y,s)$.
We again use the same scaling argument as above to get
 \[
 |\pa_{y_i}\nu(y,s)|+|y|\cdot|\nabla_y\pa_{y_i}\nu(y,s)|
 \lesssim
 \frac{R^4\lambda^\gamma}{|y|^{\frac{7}{2}}}
 \qquad \text{for } (y,s)\in B_{4R}\setminus B_{2M}\times(0,\infty).
 \]
We finally consider the equation for $\pa_{y_j}\pa_{y_i}\nu(y,s)$ and obtain
 \[
 |\pa_{y_j}\pa_{y_i}\nu(y,s)|+|y|\cdot|\nabla_y\pa_{y_j}\pa_{y_i}\nu(y,s)|
 \lesssim
 \frac{R^4\lambda^\gamma}{|y|^{\frac{9}{2}}}
 \qquad \text{for } (y,s)\in B_{2R}\setminus B_{2M}\times(0,\infty).
 \]
Since the constant $M$ is independent of $R$,
the proof is completed.
\end{proof}
We now put
 \[
 \epsilon
 =
 -H_yE
 =
 -H_y\left( E_1+\nu+{\sf c}(s)\psi_1^{(8R)} \right).
 \]
Since $-H_yg=G_\text{in}$ for $|y|<2R$,
it is clear that $\epsilon(y,s)$ satisfies
 \[
 \begin{cases}
 \pa_s\epsilon
 =
 H_y\epsilon
 +
 G_\text{in}
 &
 \text{in } B_{2R}\times(0,\infty),
 \\
 \dis
 \epsilon={\sf d_\text{in}}\psi_1^{(8R)}
 &
 \text{for } s=0.
 \end{cases}
 \]
From Lemma \ref{Lem6.2} - Lemma \ref{Lem6.4} and \eqref{6.13},
we conclude
 \begin{align}\label{6.16}
 |\epsilon(y,s)|+|y|\cdot|\nabla_y\epsilon(y,s)|
 &\lesssim
 \frac{\lambda^\gamma}{1+|y|}
 +
 \frac{R\lambda^\gamma}{1+|y|^2}
 +
 \frac{R^4 \lambda^\gamma}{1+|y|^{\frac{9}{2}}}
 +
 \frac{R\lambda^\gamma e^{-\sqrt{|\mu_1|}\cdot|y|}}{1+|y|^2}
 \nonumber
 \\
 &\lesssim
 \frac{R^4 \lambda^\gamma}{1+|y|^{\frac{9}{2}}}
 \qquad \text{for } (y,s)\in B_{2R}\times(0,\infty).
 \end{align}

\section{Outer solution}
\label{Sec7}
We now handle the outer solution $W(x,t)$.
A goal of this section is to show $W(x,t)\in X_\sigma$.
We recall that $W(x,t)\in X_\sigma$ is defined by
 \[
 W(x,t)
 \leq
 \begin{cases}
 \delta_0(T-t)^l\left( 1+|z|^{2l+2} \right)
 & \text{for } |z|<(T-t)^{-\frac{l}{2l+2}},
 \\[1mm] \dis
 \frac{\delta_0}{1+|x|^2} & \text{for } |z|>(T-t)^{-\frac{l}{2l+2}},
 \end{cases}
 \qquad
 z=\frac{x}{\sqrt{T-t}}.
 \]
The case $l=0$ is treated in \cite{delPino3}.
We here derive more elaborate decay estimates for the case $l\geq1$
by using the method in \cite{Herrero1,Herrero2,Mizoguchi,Seki}.
Throughout this section,
$\tilde{w}(x,t)\in C(\R^5\times[0,T])$
represents an extension of $w(x,t)\in X_\sigma$ defined in Section \ref{Sec5.2},
and
$(\lambda(t),\epsilon(y,t))$ represents a pair of functions obtained
in Section \ref{Sec6}.

\subsection{Choice of parameters}
\label{Sec7.1}
In this section,
we consider
 \begin{equation}\label{7.1}
 \begin{cases}
 W_t
 =
 \Delta_xW+G_\text{out}(\lambda,\tilde{w},\epsilon)
 &
 \text{in } \R^5\times(0,T),
 \\
 W = ({\bf d}\cdot {\bf e})\chi_\text{out}
 &
 \text{for } t=0,
 \end{cases}
 \end{equation}
where
${\bf d}
 =({\sf d}_0,{\sf d}_1,\cdots,{\sf d}_l)\in\R^{l+1}$
is a parameter and
 \[
 {\bf e}=(e_0(z),e_1(z),\cdots,e_l(z)),
 \qquad
 z=\frac{x}{\sqrt{T-t}}.
 \]
We recall that $e_i(z)$ is the eigenfunction defined in Section \ref{Sec3.5}.
We introduce a self-similar transformation.
 \[
 \varphi(z,\tau)=W(z\sqrt{T-t},t),
 \hspace{10mm}
 T-t=e^{-\tau}.
 \]
The function $\varphi(z,\tau)$ solves
 \begin{equation}\label{7.2}
 \begin{cases}
 \varphi_\tau
 =
 A_z\varphi+e^{-\tau}G_\text{out}(\lambda,\tilde{w},\epsilon)
 &
 \text{in } \R^5\times(\tau_0,\infty),
 \\
 \varphi = ({\bf d}\cdot{\bf e})\chi_\text{out}
 &
 \text{for } \tau=\tau_0=-\log T.
 \end{cases}
 \end{equation}
We decompose the initial data to the subspace $Y_l=$span$\{e_0,e_1,\cdots,e_l\}$
and its orthogonal complement in $L_\rho^2(\R^5)$.
 \[
 ({\bf d}\cdot{\bf e})\chi_\text{out}
 =
 \sum_{j,k=0}^l\frac{{\sf d}_j}{\kappa_k^2}(e_j\chi_\text{out},e_k)e_k
 +
 \{({\bf d}\cdot{\bf e})\chi_\text{out}\}^\bot,
 \qquad
 \kappa_j=\sqrt{(e_j,e_j)_\rho}.
 \]
We define $\Phi(z,\tau)$ as
 \[
 \varphi={\bf b}(\tau)\cdot{\bf e}+\Phi,
 \]
where
${\bf b}(\tau)=({\sf b}_0(\tau),{\sf b}_1(\tau),\cdots,{\sf b}_l(\tau))\in\R^{l+1}$.
We easily see that $\Phi(z,\tau)$ satisfies
 \begin{equation}\label{7.3}
 \begin{cases}
 \dis
 \Phi_\tau
 =
 A_z\Phi+e^{-\tau}G_\text{out}(\lambda,\tilde{w},\epsilon)
 -
 \sum_{k=0}^lk{\sf b}_ke_k
 -
 \sum_{k=0}^l
 \frac{d{\sf b}_k}{d\tau}e_k
 &\dis
 \text{in } \R^5\times(\tau_0,\infty),
 \\ \dis
 \Phi =
 \left(
 \sum_{j,k=0}^l
 \frac{{\sf d}_j}{\kappa_k^2}(e_j,\chi_\text{out}e_k)
 -
 \sum_{k=0}^l
 {\sf b}_k(\tau_0)
 \right)e_k
 +\{({\bf d}\cdot{\bf e})\chi_\text{out}\}^\bot
 &
 \text{for } \tau=\tau_0.
 \end{cases}
 \end{equation}
To obtain a solution $\Phi(z,\tau)$ satisfying
$\|\Phi(\tau)\|_\rho=o(e^{-l\tau})$,
we choose ${\bf b}(\tau)$ as
 \[
 {\sf b}_k(\tau)
 =
 -e^{-k\tau}
 \int_\tau^\infty
 \frac{e^{(k-1)\tau'}}{\kappa_k^2}
 (G_\text{out}(\lambda,\tilde{w},\epsilon),e_k)_\rho
 d\tau'.
 \]
From Lemma \ref{Lem7.1},
we verify that
 \begin{align*}
 |(G_\text{out}&(\lambda,\tilde{w},\epsilon),e_k)_\rho|
 \lesssim
 |(G_\text{out}(\lambda,\tilde{w},\epsilon),e_k)_{L_\rho^2(|z|<1)}|
 +
 e^{-2l\tau}
 \\
 &\lesssim
 \frac{e^{-l\tau}}{\lambda^2R^\frac{1}{4}}
 \left(
 {\bf 1}_{|y|<2R}
 +
 \frac{{\bf 1}_{|y|>2R}}{|y|^\frac{11}{4}},
 |e_k|
 \right)_{L_\rho^2(|z|<1)}
 +
 e^{-pl\tau}
 +
 e^{-2l\tau}
 \\
 &\lesssim
 R^\frac{19}{4}\lambda^3e^{-(l-\frac{5}{2})\tau}
 +
 R^{-\frac{1}{4}}\lambda^\frac{3}{4}
 e^{-(l-\frac{11}{8})\tau}
 +
 e^{-2l\tau}
 \\
 &\lesssim
 e^{-2l\tau}.
 \end{align*}
This implies
 \begin{equation}\label{7.4}
 |{\sf b}_k(\tau)|
 \lesssim
 e^{-(2l+1)\tau}.
 \end{equation}
From definition,
the parameter ${\bf b}(\tau)$ gives a solution of
 \begin{align*}
 \frac{d{\sf b}_k}{d\tau}
 =
 -k{\sf b}_k
 +
 \frac{e^{-\tau}}{\kappa_k^2}
 (G_\text{out}(\lambda,\tilde{w}(x,t),\epsilon(y,t)),e_k)_\rho.
 \end{align*}
We take
${\bf d}=({\sf d}_0,{\sf d}_1,\cdots,{\sf d}_l)$ as
 \[
 (\text{id}+{\sf C}){\bf d}={\bf b}(\tau_0),
 \]
where ${\sf C}$ is a constant $(l+1)\times(l+1)$ matrix defined by
 \[
 {\sf C}_{kj}=\frac{1}{\kappa_k^2}(e_j,(1-\chi_\text{out})|_{\tau=\tau_0}e_k)_\rho.
 \]
Since $\chi_\text{out}=\chi(\frac{|z|}{e^{B\tau}})$ with
$B=\frac{l+\frac{1}{2}}{2l+2}$,
we easily see that $|{\sf C}_{kj}|\ll1$,
Therefore we get from \eqref{7.4} that
 \begin{equation}\label{7.5}
 |{\bf d}|
 \lesssim
 |{\bf b}(\tau_0)|
 \lesssim
 e^{-(2l+1)\tau_0}.
 \end{equation}
By the choice of ${\bf d}$ and ${\bf b}(\tau)$,
the equation\eqref{7.3} is rewritten as
 \begin{equation}\label{7.6}
 \begin{cases}
 \dis
 \Phi_\tau
 =
 A_z\Phi+e^{-\tau}G_\text{out}^\bot
 &\dis
 \text{in } \R^5\times(\tau_0,\infty),
 \\ \dis
 \Phi =
 \{({\bf d}\cdot{\bf e})\chi_\text{out}\}^\bot
 &
 \text{for } \tau=\tau_0,
 \end{cases}
 \end{equation}
where
 $G_\text{out}^\bot
 =
 G_\text{out}(\lambda,\tilde{w},\epsilon)
 -
 \sum_{k=0}^l
 \frac{e^{-\tau}}{\kappa_k^2}
 (G_\text{out}(\lambda,\tilde{w},\epsilon),e_k)_\rho e_k$.

\subsection{Estimate of $G_\text{out}$}
\label{Sec7.2}
We here provide the estimate of $G_\text{out}$.
From \eqref{5.2}, \eqref{5.4} and \eqref{5.7},
we recall that
\begin{align*}
 G_\text{out}(\lambda,\tilde{w},\epsilon)
 &=
 h_\text{out}+h_\text{in}
 +
 \frac{1}{\lambda^2}(1-\chi_\text{in})V(y)(\Theta\chi_\text{out}+\tilde{w})
 +
 \frac{\lambda_t}{\lambda^{\frac{5}{2}}}
 (1-\chi_\text{in})\Lambda_y{\sf Q}(y)
 +
 {\sf N}(v),
 \end{align*}
 where
 \begin{align*}
 &h_\text{out}
 =
 2\nabla_x\Theta\cdot\nabla_x\chi_\text{out}
 +
 \Theta\Delta_x\chi_\text{out}
 -
 \Theta\pa_t\chi_\text{out},
 \\
 &h_\text{in}
 =
 \frac{1}{\lambda^{\frac{7}{2}}}
 \left(
 2\nabla_y\epsilon\cdot\nabla_y\chi_\text{in}+\epsilon\Delta_y\chi_\text{in}
 \right)
 +
 \frac{\lambda_t}{\lambda^{\frac{5}{2}}}\Lambda_y\epsilon\chi_\text{in}
 -
 \frac{1}{\lambda^\frac{3}{2}}\epsilon\pa_t\chi_\text{in},
 \\
 &\chi_\text{in}
 =
 \chi\left( \frac{|y|}{R} \right),
 \\
 &\chi_\text{out}
 =
 \chi\left( \frac{|z|}{e^{B\tau}} \right) \quad \text{with} \quad
 B=\frac{l+\frac{1}{2}}{2l+2},
 \\
 &{\sf N}(v)
 =
 f({\sf Q}_\lambda+\Theta\chi_\text{out}+\epsilon_\lambda\chi_\text{in}+\tilde{w})
 -
 f({\sf Q}_\lambda)
 -
 f'({\sf Q}_\lambda)
 (\Theta\chi_\text{out}+\epsilon_\lambda\chi_\text{in}+\tilde{w})
 \end{align*}
 and
 \[
 y=\frac{x}{\lambda},
 \hspace{10mm}
 z=\frac{x}{\sqrt{T-t}},
 \hspace{10mm}
 \lambda\sim \alpha_l^2(T-t)^{2l+2}.
 \]
Let ${\bf 1}_{z\in\Omega}(z)$ be a function on $\R^5$ defined by
${\bf 1}_{z\in\Omega}(z)=1$ if $z\in\Omega$
and
${\bf 1}_{z\in\Omega}(z)=0$ if $z\not\in\Omega$.
 \begin{lem}\label{Lem7.1}
 Let $w\in X_\sigma$ and
 $(\tilde w(x,t),\lambda(t),\epsilon(y,t))$ be given in Section {\rm\ref{Sec6}}.
 Then
 \begin{align*}
 |G_\text{\rm out}& (\lambda,\tilde{w},\epsilon)|
 \lesssim
 \begin{cases}
 \dis
 \quad
 \frac{e^{-l\tau}}{\lambda^2}
 \frac{1}{R^\frac{1}{4}}
 \left(
 \frac{{\bf 1}_{|y|<2R}}{1+|y|^\frac{9}{4}}
 +
 \frac{1}{1+|y|^\frac{11}{4}}
 \right)
 +
 e^{-pl\tau}
 &
 \text{\rm for}\ |z|<1,
 \\[4mm]
 \quad
 \dis
 e^{-2l\tau}
 |z|^{4l+4}
 &
 \text{\rm for}\ 1<|z|<e^{\frac{l\tau}{2l+2}},
 \\[2mm]
 \quad
 \dis
 e^{-(l-\frac{1}{2l+2})\tau}
 |z|^{2l+2}
 {\bf 1}_{e^\frac{(l+\frac{1}{2})\tau}{2l+2}<|z|<2e^\frac{(l+\frac{1}{2})\tau}{2l+2}}
 +
 1
 &
 \text{\rm for}\ e^{\frac{l\tau}{2l+2}}<|z|<e^{\frac{\tau}{2}},
 \\[4mm]
 \quad
 \dis
 \frac{e^{-(3l+2)\tau}}{|x|^3}
 +
 \frac{\delta_0^2}{|x|^{2p}}
 &
 \text{\rm for}\ |x|>1.
 \end{cases}
 \end{align*}
 \end{lem}
 \begin{proof}
Since $\Theta(x,t)=-e^{-l\tau}e_l(z)$ and $B=e^\frac{l+\frac{1}{2}}{2l+2}$,
we see that
 \begin{align}\label{7.7}
 |h_\text{out}|
 &\lesssim
 \left| \frac{\pa z_j}{\pa x_i} \right|
 \cdot
 |\nabla_z\Theta\cdot\nabla_z\chi_\text{out}|
 +
 \left| \frac{\pa z_j}{\pa x_i} \right|^2
 \cdot
 |\Theta\Delta_z\chi_\text{out}|
 +
 \left| \frac{d\tau}{dt} \right|
 \cdot
 |\Theta\pa_\tau\chi_\text{out}|
 \nonumber
 \\ 
 &\lesssim
 e^\tau e^{-l\tau}|z|^{2l}
 {\bf 1}_{e^{B\tau}<|z|<2e^{B\tau}}
 \nonumber
 \\
 &\lesssim
 e^{-(l-\frac{1}{2l+2})\tau}|z|^{2l+2}
 {\bf 1}_{e^{B\tau}<|z|<2e^{B\tau}}.
 \end{align}
We next estimate $h_\text{in}$.
Since
$\frac{e^{-(4l+3)\tau}}{\lambda^\frac{3}{2}}\sim e^{-l\tau}$
and
$|\lambda\lambda_t|\sim e^{-(4l+3)\tau}$
(see Lemma \ref{Lem6.1}),
we get from \eqref{6.16} that
 \begin{align*}
 |h_\text{in}|
 &\lesssim
 \frac{1}{\lambda^{\frac{7}{2}}}
 \left(
 \frac{|\nabla_y\epsilon|}{R}
 +
 \frac{|\epsilon|}{R^2}
 \right)
 {\bf 1}_{R<|y|<2R}
 +
 \frac{|\lambda_t|}{\lambda^\frac{5}{2}}
 |\Lambda_y\epsilon|\chi_\text{in}
 +
 \frac{1}{\lambda^\frac{3}{2}}
 |\epsilon|
 \frac{\lambda_t}{\lambda}
 {\bf 1}_{R<|y|<2R}
 \nonumber
 \\
 &\lesssim
 \frac{1}{\lambda^2}
 \frac{1}{\lambda^{\frac{3}{2}}}
 \left(
 \frac{|\nabla_y\epsilon|}{R}
 +
 \frac{|\epsilon|}{R^2}
 +
 |\lambda\lambda_t|\cdot|\epsilon|
 \right)
 {\bf 1}_{R<|y|<2R}
 +
 \frac{1}{\lambda^2}
 \frac{|\lambda\lambda_t|}{\lambda^\frac{3}{2}}
 |\Lambda_y\epsilon|\chi_\text{in}
 \nonumber
 \\
 &\lesssim
 \frac{e^{-l\tau}}{\lambda^2}
 \left(
 \frac{1}{R^\frac{5}{2}}
 +
 \frac{e^{-(4l+3)\tau}}{R^\frac{1}{2}}
 \right)
 {\bf 1}_{R<|y|<2R}
 +
 \frac{e^{-l\tau}}{\lambda^2}
 \frac{R^4e^{-(4l+3)\tau}}{1+|y|^\frac{9}{2}}
 {\bf 1}_{|y|<2R}.
 \end{align*}
 Therefore
 since $R=\tau_0$,
 we deduce that
 \begin{align}\label{7.8}
 |h_\text{in}|
 &\lesssim
 \frac{e^{-l\tau}}{\lambda^2}
 \frac{1}{R^\frac{5}{2}} 
 {\bf 1}_{R<|y|<2R}
 +
 \frac{e^{-l\tau}}{\lambda^2}
 \frac{e^{-\tau}}{1+|y|^\frac{9}{2}}
 {\bf 1}_{|y|<2R}
 \nonumber
 \\
 &\lesssim
 \frac{e^{-l\tau}}{\lambda^2}
 \frac{1}{R^\frac{1}{4}}
 \frac{{\bf 1}_{|y|<2R}}{1+|y|^\frac{9}{4}}. 
 \end{align}
We estimate the third term.
Since $\tilde w(x,t)$ satisfies \eqref{5.9},
we verify that
 \begin{align}\label{7.9}
 \frac{1-\chi_\text{in}}{\lambda^2}V(y)
 &
 |\Theta\chi_\text{out}+\tilde w|
 \lesssim
 \frac{e^{-l\tau}{\bf 1}_{|y|>R}}{\lambda^2|y|^4}
 \left(
 \left( 1+|z|^{2l} \right)\chi_\text{out}
 +
 \delta_0
 \left( 1+|z|^{2l+2} \right){\bf 1}_{|z|<e^{\frac{l\tau}{2l+2}}}
 \right)
 \nonumber
 \\
 &\qquad
 +
 \frac{1}{\lambda^2|y|^4}
 \frac{{\bf 1}_{|z|>\frac{l\tau}{2+2}}}{1+|x|^2}
 \nonumber
 \\
 &
 \lesssim
 \frac{e^{-l\tau}}{\lambda^2|y|^4}
 \left(
 {\bf 1}_{|y|>R}
 {\bf 1}_{|z|<1}
 +
 |z|^{2l+2}
 {\bf 1}_{1<|z|<2e^\frac{(l+\frac{1}{2})\tau}{2l+2}}
 \right)
 \nonumber
 \\
 &\qquad
 +
 \frac{\lambda^2e^{2\tau}}{|z|^4}
 {\bf 1}_{e^{\frac{l\tau}{2l+2}}<|z|<e^\frac{\tau}{2}}
 +
 \frac{\lambda^2}{|x|^6}{\bf 1}_{|x|>1}
 \nonumber
 \\
 &
 \lesssim
 \frac{e^{-l\tau}}{\lambda^2|y|^4}
 {\bf 1}_{|y|>R}
 {\bf 1}_{|z|<1}
 +
 e^{-(5l+2)\tau}
 |z|^{2l-2}
 {\bf 1}_{1<|z|<e^\frac{(l+\frac{1}{2})\tau}{2l+2}}
 \nonumber
 \\
 &\qquad
 +
 e^{-(4l+4-\frac{2}{l+1})\tau}
 {\bf 1}_{e^{\frac{l\tau}{2l+2}}<|z|<e^\frac{\tau}{2}}
 +
 \frac{e^{-(4l+4)\tau}}{|x|^6}{\bf 1}_{|x|>1}.
 \end{align}
The fourth term is easily estimated as
 \begin{align}\label{7.10}
 \left|
 \frac{\lambda_t}{\lambda^\frac{5}{2}}(1-\chi_\text{in})\Lambda_y{\sf Q}(y)
 \right|
 &<
 \frac{1}{\lambda^2}\frac{|\lambda_t|}{\lambda^\frac{3}{2}}
 \frac{1}{|y|^3}{\bf 1}_{|y|>R}
 <
 \frac{e^{-l\tau}}{\lambda^2}
 \frac{1}{|y|^3}{\bf 1}_{|y|>R}
 \nonumber
 \\
 &<
 \frac{e^{-l\tau}}{\lambda^2}\frac{1}{|y|^3}
 {\bf 1}_{|y|>R}
 {\bf 1}_{|z|<1}
 +
 \lambda e^{\frac{3}{2}\tau}\frac{e^{-l\tau}}{|z|^3}
 {\bf 1}_{|z|>1}
 \nonumber
 \\
 &<
 \frac{e^{-l\tau}}{\lambda^2}\frac{1}{|y|^3}
 {\bf 1}_{|y|>R}
 {\bf 1}_{|z|<1}
 +
 e^{-(3l+\frac{1}{2})\tau}
 {\bf 1}_{1<|z|<e^\frac{\tau}{2}}
 +
 \frac{e^{-(3l+2)\tau}}{|x|^3}
 {\bf 1}_{|x|>1}.
 \end{align}
We finally estimate ${\sf N}(v)$.
Since
 \[
 |\, |a+b|^{p-1}(a+b)-|a|^{p-1}a-p|a|^{p-1}b\, |
 \lesssim
 |a|^{p-2}b^2+|b|^p
 \qquad\text{for } a,b\in\R
 \qquad(p\geq2)
 \]
we get
 \begin{align*}
 |{\sf N}(v)|
 &=
 |f({\sf Q}_\lambda+\Theta\chi_\text{out}+\epsilon_\lambda\chi_\text{in}+\tilde w)
 -
 f({\sf Q}_\lambda)
 -
 f'({\sf Q}_\lambda)
 (\Theta\chi_\text{out}+\epsilon_\lambda\chi_\text{in}+\tilde w)|
 \\
 &\lesssim
 f''({\sf Q}_{\lambda})
 (\Theta\chi_\text{out}+\epsilon_\lambda\chi_\text{in}+\tilde w)^2
 +
 f(\Theta\chi_\text{out}+\epsilon_\lambda\chi_\text{in}+\tilde w)
 \\
 &\lesssim
 \frac{1}{\sqrt{\lambda}}
 \frac{\Theta^2\chi_\text{out}+\epsilon_\lambda^2\chi_\text{in}+\tilde w^2}{1+|y|}
 +
 |\Theta|^p\chi_\text{out}
 +
 |\epsilon_\lambda|^p\chi_\text{in}
 +
 |\tilde w|^p
 \\
 &\lesssim
 \left(
 \frac{1}{\sqrt{\lambda}}
 \frac{\Theta^2+\tilde w^2}{1+|y|}
 +
 |\Theta|^p
 +
 |\tilde w|^p
 \right)
 {\bf 1}_{|z|<1}
 +
 \left(
 \frac{1}{\sqrt{\lambda}}
 \frac{\epsilon_\lambda^2}{1+|y|}
 +
 |\epsilon_\lambda|^p
 \right)
 {\bf 1}_{|z|<2R}
 \\
 &\qquad
 +
 \left(
 \frac{1}{\sqrt{\lambda}}
 \frac{\Theta^2+\tilde w^2}{|y|}
 +
 |\Theta|^p
 +
 |\tilde w|^p
 \right)
 {\bf 1}_{1<|z|<e^\frac{l\tau}{2l+2}}
 \\
 &\qquad
 +
 \left(
 \frac{1}{\sqrt{\lambda}}
 \frac{\Theta^2}{|y|}
 +
 |\Theta|^p
 \right)
 {\bf 1}_{e^\frac{l\tau}{2l+2}<|z|<2e^\frac{(l+\frac{1}{2})\tau}{2l+2}}
 +
 \left(
 \frac{1}{\sqrt{\lambda}}
 \frac{\tilde w^2}{|y|}
 +
 |\tilde w|^p
 \right)
 {\bf 1}_{|z|>e^\frac{l\tau}{2l+2}}.
 \end{align*}
Since $|\Theta(x,t)|\lesssim(T-t)^l(1+|z|^{2l})$ and
$\tilde{w}(x,t)$ satisfies \eqref{5.9},
we see that
\begin{align*}
 &
 \left(
 \frac{1}{\sqrt{\lambda}}
 \frac{\Theta^2+\tilde w^2}{1+|y|}
 +
 |\Theta|^p
 +
 |\tilde w|^p
 \right)
 {\bf 1}_{|z|<1}
 \lesssim
 \left(
 \frac{1}{\sqrt{\lambda}}
 \frac{e^{-2l\tau}}{1+|y|}
 +
 e^{-pl\tau}
 \right)
 {\bf 1}_{|z|<1}
 \\
 &\qquad\lesssim
 \frac{e^{-2l\tau}}{\sqrt{\lambda}}
 \left(
 {\bf 1}_{|y|<1}
 +
 \frac{|y|^\frac{7}{4}}{|y|^\frac{11}{4}}
 {\bf 1}_{|y|>1}
 {\bf 1}_{|z|<1}
 \right)
 +
 e^{-pl\tau}
 {\bf 1}_{|z|<1}.
 \end{align*}
Here we note that
$|y|^\frac{7}{4}{\bf 1}_{|z|<1}
 =(\lambda^{-1}e^{-\frac{\tau}{2}})^\frac{7}{4}|z|^\frac{7}{4}{\bf 1}_{|z|<1}
 <(\lambda^{-1}e^{-\frac{\tau}{2}})^\frac{7}{4}{\bf 1}_{|z|<1}$.
Therefore we deduce that
\begin{align*}
 \left(
 \frac{1}{\sqrt{\lambda}}
 \frac{\Theta^2+\tilde w^2}{1+|y|}
 +
 |\Theta|^p
 +
 |\tilde w|^p
 \right)
 {\bf 1}_{|z|<1}
 &\lesssim
 \frac{e^{-2l\tau}}{\sqrt{\lambda}}
 \left(
 1
 +
 \frac{e^{-\frac{7\tau}{8}}}{\lambda^\frac{7}{4}}\frac{1}{|y|^\frac{11}{4}}
 {\bf 1}_{|y|>1}
 \right)
 {\bf 1}_{|z|<1}
 +
 e^{-pl\tau}
 {\bf 1}_{|z|<1}
 \\
 &\lesssim
 \frac{1}{\lambda^2}
 \frac{e^{-(\frac{3l}{2}+\frac{3}{8})\tau}}{1+|y|^\frac{11}{4}}
 {\bf 1}_{|z|<1}
 +
 e^{-pl\tau}
 {\bf 1}_{|z|<1}.
 \end{align*}
Furthermore from \eqref{6.16} and $\tau_0=R$,
we verify that
 \begin{align*}
 \left(
 \frac{1}{\sqrt{\lambda}}
 \frac{\epsilon_\lambda^2}{1+|y|}
 +
 |\epsilon_\lambda|^p
 \right)
 {\bf 1}_{|y|<2R}
 &\lesssim
 \frac{1}{\lambda^2}
 \left(
 \frac{\lambda^\frac{3}{2}}{1+|y|}
 \frac{R^8e^{-2l\tau}}{1+|y|^9}
 +
 \frac{\lambda^2R^{4p}e^{-pl\tau}}{1+|y|^\frac{9p}{2}}
 \right)
 {\bf 1}_{|y|<2R}
 \\
 &\lesssim
 \frac{e^{-l\tau}}{\lambda^2}
 \left(
 \frac{R^8e^{-(4l+3)\tau}}{1+|y|^{10}}
 +
 \frac{R^{4p}e^{-((p+3)l+4)\tau}}{1+|y|^{\frac{9p}{2}}}
 \right)
 {\bf 1}_{|y|<2R}
 \\
 &\lesssim
 \frac{e^{-l\tau}}{\lambda^2}
 \left(
 \frac{e^{-\tau}}{1+|y|^{10}}
 +
 \frac{e^{-\tau}}{1+|y|^{\frac{9p}{2}}}
 \right)
 {\bf 1}_{|y|<2R}.
 \end{align*}
For the case $2<p<3$,
there exists $c_p>0$ such that
 \begin{equation}\label{7.11}
 |\xi|^p<c_p(|\xi|^2+|\xi|^3) \qquad\text{for } \xi\in\R.
 \end{equation}
From this relation,
we see that
 \begin{align*}
 &
 \left(
 \frac{1}{\sqrt{\lambda}}
 \frac{\Theta^2+\tilde w^2}{|y|}
 +
 |\Theta|^p
 +
 |\tilde w|^p
 \right)
 {\bf 1}_{1<|z|<e^\frac{l\tau}{2l+2}}
 \\
 &\qquad\lesssim
 \left(
 \frac{1}{\sqrt{\lambda}}
 \frac{e^{-2l\tau}(|z|^{4l}+\delta_0^2|z|^{4l+4})}{|y|}
 +
 e^{-pl\tau}\left( |z|^{2pl}+\delta_0^p|z|^{2p(l+1)} \right)
 \right)
 {\bf 1}_{1<|z|<e^\frac{l\tau}{2l+2}}
 \\
 &\qquad\lesssim
 \left(
 \frac{1}{\sqrt{\lambda}}
 \frac{e^{-2l\tau}|z|^{4l+4}}{|y|}
 +
 \left( e^{-l\tau}|z|^{2(l+1)} \right)^p
 \right)
 {\bf 1}_{1<|z|<e^\frac{l\tau}{2l+2}}
 \\
 &\qquad\lesssim
 \left(
 \frac{\sqrt{\lambda}e^\frac{\tau}{2}}{|z|}
 e^{-2l\tau}|z|^{4l+4}
 +
 e^{-2l\tau}
 |z|^{4(l+1)}
 +
 e^{-3l\tau}
 |z|^{6(l+1)}
 \right)
 {\bf 1}_{1<|z|<e^\frac{l\tau}{2l+2}}
 \\
 &\qquad\lesssim
 e^{-2l\tau}
 |z|^{4l+4}
 {\bf 1}_{1<|z|<e^\frac{l\tau}{2l+2}}.
 \end{align*}
We use \eqref{7.11} again to get
 \begin{align*}
 \left(
 \frac{1}{\sqrt{\lambda}}
 \frac{\Theta^2}{|y|}
 +
 |\Theta|^p
 \right)
 {\bf 1}_{e^\frac{l\tau}{2l+2}<|z|<2e^\frac{(l+\frac{1}{2})\tau}{2l+2}}
 &\lesssim
 \left(
 \frac{1}{\sqrt{\lambda}}
 \frac{e^{-2l\tau}|z|^{4l}}{|y|}
 +
 e^{-pl\tau}|z|^{2pl}
 \right)
 {\bf 1}_{e^\frac{l\tau}{2l+2}<|z|<2e^\frac{(l+\frac{1}{2})\tau}{2l+2}}
 \\
 &
 \lesssim
 \left(
 e^{-2l\tau}|z|^{4l}
 +
 e^{-3l\tau}|z|^{6l}
 \right)
 {\bf 1}_{e^\frac{l\tau}{2l+2}<|z|<2e^\frac{(l+\frac{1}{2})\tau}{2l+2}}
 \\
 &
 \lesssim
 e^{-2l\tau}|z|^{4l}
 {\bf 1}_{e^\frac{l\tau}{2l+2}<|z|<2e^\frac{(l+\frac{1}{2})\tau}{2l+2}}
 \\
 &
 \lesssim
 e^{-(l+1)\tau}|z|^{2l+2}
 {\bf 1}_{e^\frac{l\tau}{2l+2}<|z|<2e^\frac{(l+\frac{1}{2})\tau}{2l+2}}.
 \end{align*}
We finally estimate the last term.
 \begin{align*}
 \left(
 \frac{1}{\sqrt{\lambda}}
 \frac{\tilde w^2}{|y|}
 +
 |\tilde w|^p
 \right)
 {\bf 1}_{|z|>e^\frac{l\tau}{2l+2}}
 &\lesssim
 \left(
 \frac{1}{\sqrt{\lambda}}
 \frac{1}{|y|}
 \frac{\delta_0^2}{1+|x|^4}
 +
 \frac{\delta_0^p}{1+|x|^{2p}}
 \right)
 {\bf 1}_{|z|>e^\frac{l\tau}{2l+2}}
 \\
 &\lesssim
 \left(
 \frac{\sqrt{\lambda}e^\frac{\tau}{2}}{|z|}
 +
 1
 \right)
 {\bf 1}_{e^\frac{l\tau}{2l+2}<|z|<e^\frac{\tau}{2}}
 +
 \left(
 \frac{\sqrt{\lambda}}{|x|}
 \frac{\delta_0^2}{|x|^4}
 +
 \frac{\delta_0^p}{|x|^{2p}}
 \right)
 {\bf 1}_{|x|>1}
 \\
 &\lesssim
 {\bf 1}_{e^\frac{l\tau}{2l+2}<|z|<e^\frac{\tau}{2}}
 +
 \frac{\delta_0^2}{|x|^{2p}}
 {\bf 1}_{|x|>1}.
 \end{align*}
From the above estimates,
we conclude that
 \begin{align}\label{7.12}
 |{\sf N}(v)|
 &\lesssim
 \frac{e^{-l\tau}}{\lambda^2}
 \frac{e^{-\frac{3}{8}\tau}}{1+|y|^\frac{11}{4}}
 {\bf 1}_{|z|<1}
 +
 e^{-pl\tau}
 {\bf 1}_{|z|<1}
 +
 e^{-2l\tau}
 |z|^{4l+4}
 {\bf 1}_{1<|z|<e^\frac{l\tau}{2l+2}}
 \nonumber
 \\
 &\qquad
 +
 e^{-(l+1)\tau}|z|^{2l+2}
 {\bf 1}_{e^\frac{l\tau}{2l+2}<|z|<2e^\frac{(l+\frac{1}{2})\tau}{2l+2}}
 +
 {\bf 1}_{e^\frac{l\tau}{2l+2}<|z|<e^\frac{\tau}{2}}
 +
 \frac{\delta_0^2}{|x|^{2p}}
 {\bf 1}_{|x|>1}.
 \end{align}
Combining \eqref{7.7} - \eqref{7.10} and \eqref{7.12},
we complete the proof.
 \end{proof}

\subsection{$L_\rho^2(\R^5)$ estimate for $\Phi(z,\tau)$}
\label{Sec7.3}
To derive the estimate for a solution $\Phi(z,\tau)$ of \eqref{7.6},
we first consider
 \[
 \begin{cases}
 \dis
 \pa_\tau\Phi_1
 =
 A_z\Phi_1
 &\dis
 \text{in } \R^5\times(\tau_0,\infty),
 \\ \dis
 \Phi_1 =
 \{({\bf d}\cdot{\bf e})\chi_\text{out}\}^\bot
 &
 \text{for } \tau=\tau_0.
 \end{cases}
 \]
 \begin{lem}\label{Lem7.2}
 There exists $K_1>1$ independent of $R$, $\delta_0$, $\sigma$ such that
 \[
 |\Phi_1(z,\tau)|
 <
 K_1e^{-(l+1)\tau}
 \left( 1+|z|^{2l+2} \right)
 \qquad\text{\rm for}\ (z,\tau)\in\R^5\times(\tau_0,\infty).
 \]
 \end{lem}
\begin{proof}
We estimate the initial data.
 \begin{align*}
 \{({\bf d}\cdot{\bf e})\chi_\text{out}\}^\bot
 &=
 ({\bf d}\cdot{\bf e})\chi_\text{out}
 -
 \sum_{j,k=0}^l
 \frac{{\sf d}_j}{\kappa_k^2}(e_j\chi_\text{out},e_k)_\rho e_k
 \\
 &=
 ({\bf d}\cdot{\bf e})\chi_\text{out}
 -
 \sum_{j,k=0}^l
 \frac{{\sf d}_j}{\kappa_k^2}(e_j,e_k)_\rho e_k
 +
 \sum_{j,k=0}^l
 \frac{{\sf d}_j}{\kappa_k^2}(e_j(1-\chi_\text{out}),e_k)_\rho e_k
 \\
 &=
 ({\bf d}\cdot{\bf e})(\chi_\text{out}-1)
 +
 \sum_{j,k=0}^l
 \frac{{\sf d}_j}{\kappa_k^2}(e_j(1-\chi_\text{out}),e_k)_\rho e_k.
 \end{align*}
Since
$\chi_\text{out}=\chi(\frac{|z|}{e^{B\tau}})$ with $B=\frac{l+\frac{1}{2}}{2l+2}$,
we see from \eqref{7.5} that
 \[
 \| \{({\bf d}\cdot{\bf e})\chi_\text{out}\}^\bot \|_\rho
 \lesssim
 |{\bf d}|
 \lesssim
 e^{-(2l+1)\tau_0}.
 \]
Therefore we deduce that
 \begin{equation}\label{7.13}
 \|\Phi_1(\tau)\|_\rho
 \lesssim
 \| \{({\bf d}\cdot{\bf e})\chi_\text{out}\}^\bot \|_\rho
 e^{-(l+1)(\tau-\tau_0)}
 \lesssim
 e^{-(l+1)\tau}.
 \end{equation}
We next derive a pointwise estimate.
Let $e_{l+1}(z)$ be given in Section \ref{Sec3.5},
which is written as
 \[
 e_{l+1}(z)=1+a_1|z|^2+a_2|z|^4+\cdots+a_{l+1}|z|^{2l+2}.
 \]
To construct a comparison function,
we define
 \[
 \tilde e_{l+1}(z)=ke_{l+1}(z)
 \qquad \text{with}\quad
 k=
 \begin{cases}
 1 & \text{if } a_{l+1}>0,
 \\
 -1 & \text{if } a_{l+1}<0.
 \end{cases}
 \]
From this definition,
there exists $r_{l+1}>0$ such that
 \[
 \tilde e_{l+1}(z)>\frac{|a_{l+1}|}{2}|z|^{2l+2}
 \qquad\text{for } |z|>r_{l+1}.
 \]
Therefore
we note from \eqref{7.13} that
there exists $K>1$ such that
$|\Phi_1(z,\tau)|<Ke^{-(l+1)\tau}\tilde e_{l+1}(z)$ for $|z|=r_{l+1}$.
Furthermore it holds from \eqref{7.5} that
$|\Phi_1(\tau_0)|<|{\bf d}|\cdot|{\bf e}(z)|\lesssim e^{-(2l+1)\tau_0}|z|^{2l}$
for $|z|>r_{l+1}$.
Therefore
a comparison argument shows that there exists $K'>1$ such that
 \[
 |\Phi_1(z,\tau)|
 <
 K'e^{-(l+1)\tau}\tilde e_{l+1}(z)
 \qquad\text{for } |z|>r_{l+1},\ \tau>\tau_0.
 \]
This completes the proof.
\end{proof}
Next we write $\Phi$ as
 \[
 \Phi=\Phi_1+\Phi_2.
 \]
The function $\Phi_2(z,\tau)$ solves
 \begin{equation}\label{7.14}
 \begin{cases}
 \dis
 \pa_\tau\Phi_2
 =
 A_z\Phi_2+e^{-\tau}G_\text{out}^\bot
 &\dis
 \text{in } \R^5\times(\tau_0,\infty),
 \\ \dis
 \Phi_2 = 0
 &
 \text{for } \tau=\tau_0.
 \end{cases}
 \end{equation}
We first provide $L_\rho^2$ estimates of $G_\text{out}$.
 \begin{lem}\label{Lem7.3}
 It holds that
 \[
 \|G_\text{\rm out}\|_\rho
 \lesssim
 e^{-(2l-\frac{1}{2})\tau}.
 \]
 \end{lem}
\begin{proof}
 From Lemma \ref{Lem7.1},
 we easily verify that
 \begin{align*}
 \|G_\text{out}\|_\rho
 &\lesssim
 \|G_\text{out}{\bf 1}_{|z|<1}\|_\rho
 +
 \|G_\text{out}{\bf 1}_{|z|>1}\|_\rho
 \\
 &\lesssim
 \|G_\text{out}{\bf 1}_{|z|<1}\|_\rho
 +
 e^{-2l\tau}.
 \end{align*}
We estimate the first term.
 \begin{align*}
 \|G_\text{out}{\bf 1}_{|z|<1}\|_\rho
 &\lesssim
 \frac{e^{-l\tau}}{\lambda^2R^\frac{1}{4}}
 \left(
 \left\|
 \frac{{\bf 1}_{|y|<2R}}{1+|y|^\frac{9}{4}}
 \right\|_{L^2(|z|<1)}
 +
 \left\|
 \frac{{\bf 1}_{|y|<2R}+{\bf 1}_{|y|>2R}}{1+|y|^\frac{11}{4}}
 \right\|_{L^2(|z|<1)}
 \right)
 +
 e^{-pl\tau}
 \\
 &\qquad\lesssim
 \frac{e^{-l\tau}}{\lambda^2R^\frac{1}{4}}
 \left(
 \left( R\lambda e^\frac{\tau}{2} \right)^\frac{5}{2}
 +
 \left\|
 \frac{{\bf 1}_{|y|>2R}}{|y|^\frac{11}{4}}
 \right\|_{L^2(|z|<1)}
 \right)
 +
 e^{-pl\tau}
 \\
 &\qquad\lesssim
 \frac{e^{-l\tau}}{\lambda^2R^\frac{1}{4}}
 \left(
 \left( R\lambda e^\frac{\tau}{2} \right)^\frac{5}{2}
 +
 \left( \lambda e^\frac{\tau}{2} \right)^\frac{11}{4}
 \left( R\lambda e^\frac{\tau}{2} \right)^{-\frac{1}{4}}
 \right)
 +
 e^{-pl\tau}
 \\
 &\qquad\lesssim
 R^\frac{9}{4}e^{-\frac{\tau_0}{4}}e^{-(2l-\frac{1}{2})\tau}
 +
 e^{-pl\tau}.
 \end{align*}
Since $\tau_0=R$, we complete the proof.
\end{proof}
From this estimates,
we immediately obtain $L_\rho^2$ estimates of $\Phi_2(z,\tau)$. 
 \begin{lem}\label{Lem7.4}
 There exists $K_2>1$ independent of $R$, $\delta_0$, $\sigma$ such that
 \[
 \|\Phi_2(\tau)\|_\rho<K_2e^{-(l+1)\tau}.
 \]
 \end{lem}
 \begin{proof}
We take the inner product $(\cdot,\Phi_2)_\rho$ in \eqref{7.14} to get
 \begin{align*}
 \frac{1}{2}\frac{d}{d\tau}\|\Phi_2\|_\rho^2
 =
 (A_z\Phi_2,\Phi_2)_\rho+e^{-\tau}(G_\text{out}^\bot,\Phi_2)_\rho.
 \end{align*}
 Since $(A_z\Phi_2,\Phi_2)<-(l+1)\|\Phi_2\|_\rho^2$,
 we deduce from Lemma \ref{Lem7.3} that
 \[
 \|\Phi_2(\tau)\|_\rho\lesssim e^{-(l+1)\tau}.
 \]
 The proof is completed.
\end{proof}

\subsection{Pointwise estimate for $\Phi_2(z,\tau)$ in $|z|<e^\frac{l\tau}{2l+2}$}
\label{Sec7.4}
Since $\Phi_2(z,\tau)$ is a solution of \eqref{7.14},
it is written as an integral form.
 \begin{equation}\label{7.15}
 \Phi_2(z,\tau)
 =
 \int_{\tau_0}^\tau
 e^{A_z(\tau-\tau')e^{-\tau'}}G_\text{out}^\bot
 d\tau'.
 \end{equation}
We here estimate $\Phi_2(z,\tau)$
by the same manner as in \cite{Herrero1,Herrero2}
(see also \cite{Mizoguchi,Seki} and Section 5.2 \cite{Harada2}).
For simplicity we put
 \[
 b=
 \min\left\{ \frac{2l+\frac{3}{2}}{4},\frac{2l+1}{2l+2} \right\}.
 \]
 \begin{pro}\label{Pro7.1}
 There exists $K_3>1$ independent of $R$, $\delta_0$, $\sigma$ such that
 \begin{align*}
 \int_{\tau_0}^\tau
 e^{A_z(\tau-\tau')}e^{-\tau'}|G_\text{\rm out}^\bot|d\tau'
 &<
 K_3\left(
 \frac{1}{R^\frac{1}{4}}
 \frac{e^{-l\tau}}{1+|y|^\frac{1}{4}}
 +
 T^be^{-l\tau}
 +
 \frac{e^{-l\tau}}{R^\frac{1}{4}}|z|^{-2l}
 +
 e^{-(2l+1)\tau}
 |z|^{4l+4}
 \right.
 \\
 &\qquad
 \left.
 +
 e^{-(l+\frac{2l+1}{2l+2})\tau}
 |z|^{2l+2}
 \right)
 \qquad\text{\rm for}\ (z,\tau)\in\R^n\times(\tau_0,\infty).
 \end{align*}
 \end{pro}
As a consequence of this proposition,
we obtain
 \begin{align}\label{7.16}
 \int_{\tau_0}^\tau
 e^{A_z(\tau-\tau')}e^{-\tau'}|G_\text{\rm out}^\bot|d\tau'
 &\lesssim
 \frac{1}{R^\frac{1}{4}}
 \frac{e^{-l\tau}}{1+|y|^\frac{1}{4}}
 +
 T^be^{-l\tau}
 +
 \frac{e^{-l\tau}}{R^\frac{1}{4}}|z|^{-2l}
 \nonumber
 \\
 &\qquad
 +
 e^{-(l+\frac{2l+1}{2l+2})\tau}
 |z|^{2l+2}
 \qquad\text{for }
 |z|<e^\frac{l\tau}{2l+2},\ \tau\in(\tau_0,\infty).
 \end{align}
For simplicity of computation,
we arrange the estimate in Lemma \ref{Lem7.1} as
 \begin{align*}
 |G_\text{\rm out}|
 &\lesssim
 \frac{e^{-l\tau}}{\lambda^2}
 \frac{1}{R^\frac{1}{4}}
 \frac{1}{1+|y|^\frac{9}{4}}
 +
 \left(
 e^{-pl\tau}
 +
 e^{-2l\tau}
 |z|^{4l+4}
 +
 e^{-(l-\frac{1}{2l+2})\tau}
 |z|^{2l+2}
 \right)
 \qquad\text{for } z\in\R^5
 \nonumber
 \\
 &=:
 G_\text{out}^{(1)}
 +
 G_\text{out}^{(2)}.
 \end{align*}
We first prepare the following lemma.
 \begin{lem}\label{Lem7.5}
 There exists $k>0$ independent of $\tau_1$, $R$, $\delta_0$, $\sigma$
 such that
 \begin{align*}
 \int_{\tau_1}^\tau
 e^{A_z(\tau-\tau')}e^{-\tau'}|G_\text{\rm out}^\bot|d\tau'
 &<
 k\left(
 \frac{1}{R^\frac{1}{4}}
 \frac{e^{-l\tau_1}}{1+|y|^\frac{1}{4}}
 +
 T^be^{-l\tau_1}
 +
 e^{-(2l+1)\tau}|z|^{4l+4}
 \right.
 \\
 &\qquad
 \left.
 +
 e^{-(l+\frac{2l+1}{2l+2})\tau}
 |z|^{2l+2}
 \right)
 \qquad\text{\rm for}\ (z,\tau)\in\R^5\times(\tau_1,\infty).
 \end{align*}
 \end{lem}
 \begin{proof}
We write the integral as
 \begin{align*}
 \int_{\tau_1}^\tau
 e^{A_z(\tau-\tau')}e^{-\tau'}|G_\text{\rm out}^\bot|d\tau'
 &\lesssim
 \int_{\tau_1}^\tau e^{A_z(\tau-\tau')}
 e^{-\tau'}
 \left(
 \left| G_\text{out}^{(1)} \right|
 +
 \left| G_\text{out}^{(2)} \right|
 \right)
 d\tau'
 \\
 &\qquad
 +
 \sum_{k=0}^l
 \int_{\tau_1}^\tau
 \left|
 e^{A_z(\tau-\tau')}
 \frac{e^{-\tau'}}{\kappa_k^2}
 (G_\text{out},e_k)_\rho e_k
 \right|
 d\tau'.
 \end{align*}
From Lemma \ref{Lem3.7},
we see that
 \begin{align*}
 \int_{\tau_1}^\tau e^{A_z(\tau-\tau')}
 e^{-\tau'}
 \left| G_\text{out}^{(2)} \right|
 d\tau'
 &\lesssim
 \int_{\tau_1}^\tau
 \left\{
 e^{-(pl+1)\tau'}
 +
 e^{-(2l+1)\tau'}
 \left(
 1+e^{-(2l+2)(\tau-\tau')}|z|^{4l+4}
 \right)
 \right.
 \\
 &\qquad
 \left.
 +
 e^{-(l+\frac{2l+1}{2l+2})\tau'}
 \left(
 1+e^{-(l+1)(\tau-\tau')}|z|^{2l+2}
 \right)
 \right\}
 d\tau'
 \\
 &\lesssim
 e^{-(l+\frac{2l+1}{2l+2})\tau_1}
 +
 e^{-(2l+1)\tau}|z|^{4l+4}
 +
 e^{-(l+\frac{2l+1}{2l+2})\tau}
 |z|^{2l+2}
 \\
 &\lesssim
 T^\frac{2l+1}{2l+2}e^{-l\tau_1}
 +
 e^{-(2l+1)\tau}|z|^{4l+4}
 +
 e^{-(l+\frac{2l+1}{2l+2})\tau}
 |z|^{2l+2}.
 \end{align*}
Since $|a|^k\lesssim1+|a|^{l+1}$ for $k=0,1,2,\cdots,l$,
we get from Lemma \ref{Lem7.3} that
 \begin{align*}
 \int_{\tau_1}^\tau
 \left|
 e^{A_z(\tau-\tau')}
 e^{-\tau'}
 (G_\text{out},e_k)_\rho e_k
 \right|
 d\tau'
 &=
 \int_{\tau_1}^\tau
 \left|
 e^{-k(\tau-\tau')}
 e^{-\tau'}
 (G_\text{out},e_k)_\rho e_k
 \right|
 d\tau'
 \\
 &\lesssim
 e^{-(2l+\frac{1}{2})\tau_1}
 \left(
 e^{-(\tau-\tau_1)}
 \left( 1+|z|^{2} \right)
 \right)^k
 \\
 &\lesssim
 e^{-(2l+\frac{1}{2})\tau_1}
 \left\{
 1+
 \left(
 e^{-(\tau-\tau_1)}
 \left( 1+|z|^{2} \right)
 \right)^{l+1}
 \right\}
 \\
 &\lesssim
 e^{-(2l+\frac{1}{2})\tau_1}
 +
 e^{-(l-\frac{1}{2})\tau_1}e^{-(l+1)\tau}|z|^{2l+2}.
 \end{align*}
Furthermore
Lemma \ref{Lem3.8} implies
 \begin{align*}
 \int_{\tau_1}^\tau e^{A_z(\tau-\tau')}
 e^{-\tau'}
 \left| G_\text{out}^{(1)} \right|
 d\tau'
 &\lesssim
 \frac{1}{R^\frac{1}{4}}
 \frac{e^{-l\tau_1}}{1+|y|^\frac{1}{4}}
 +
 T^{(2l+\frac{3}{2})\frac{1}{4}}e^{-l\tau_1}.
 \end{align*}
 The proof is completed.
 \end{proof}
To obtain the estimate in Proposition \ref{Pro7.1},
we consider four cases separately.
 \[
 \begin{array}{cll}
 \text{(i)}   & z\in\R^n \quad\text{and}\quad \tau\in(\tau_0,\tau_0+1),
 \\[1mm]
 \text{(ii)}  & |z|<4 \quad\text{and}\quad \tau\in(\tau_0+1,\infty),
 \\[1mm]
 \text{(iii)} & 2<|z|<e^\frac{\tau-\tau_0}{2} \quad\text{and}\quad \tau\in(\tau_0+1,\infty),
 \\[1mm]
 \text{(iv)} & |z|>e^\frac{\tau-\tau_0}{2} \quad\text{and}\quad
 \tau\in(\tau_0+1,\infty).
 \end{array}
 \]

\subsubsection{(i) Estimate in $z\in\R^n$ and $\tau\in(\tau_0,\tau_0+1)$}
\label{Sec7.4.1}
In this case,
the estimate follows from Lemma \ref{Lem7.5}.

\subsubsection{(ii) Estimate in $|z|<4$ and $\tau\in(\tau_0+1,\infty)$}
\label{Sec7.4.2}
We divide the integral in \eqref{7.15} into two parts.
 \[
 \Phi_2(\tau)
 =
 \int_{\tau_0}^{\tau-1}
 e^{A_z(\tau-\tau')}e^{-\tau'}G_\text{out}^\bot
 d\tau'
 +
 \int_{\tau-1}^\tau
 e^{A_z(\tau-\tau')}e^{-\tau'}G_\text{out}^\bot
 d\tau'.
 \]
We apply Lemma \ref{Lem3.6} and Lemma \ref{Lem7.3} to get
 \begin{align*}
 \left|
 \int_{\tau_0}^{\tau-1}
 e^{A_z(\tau-\tau')}e^{-\tau'}G_\text{out}^\bot
 d\tau'
 \right|
 &\lesssim
 \int_{\tau_0}^{\tau-1}
 \left|
 e^{A_z\frac{1}{2}}e^{A_z(\tau-\tau'-\frac{1}{2})}
 e^{-\tau'}G_\text{out}^\bot
 \right|
 d\tau'
 \\
 &\lesssim
 \int_{\tau_0}^{\tau-1}
 \frac{\frac{e^{-\frac{1}{2}}|z|^2}{4(1+e^{-\frac{1}{2}})}}
 {(1-e^{-\frac{1}{2}})^\frac{n}{4}}
 \left\|
 e^{A_z(\tau-\tau'-\frac{1}{2})}e^{-\tau'}G_\text{out}^\bot
 \right\|_\rho
 d\tau'
 \\
 &\lesssim
 \int_{\tau_0}^{\tau-1}
 e^{-(l+1)(\tau-\tau')}
 \left\|
 e^{-\tau'}G_\text{out}^\bot
 \right\|_\rho
 d\tau'
 \\
 &\lesssim
 \int_{\tau_0}^{\tau-1}
 e^{-(l+1)(\tau-\tau')}
 e^{-(2l+\frac{1}{2})\tau'}
 d\tau'
 \\
 &\lesssim
 e^{-(l+1)\tau}
 \qquad\text{for } |z|<4,\ \tau\in(\tau_0+1,\infty).
\end{align*}
The second integral is reduced to Lemma \ref{Lem7.5}. 

\subsubsection{(iii) Estimate in $2<|z|<e^\frac{\tau-\tau_0}{2}$
and $\tau\in(\tau_0+1,\infty)$}
\label{Sec7.4.3}
For $2<|z|<e^\frac{\tau-\tau_0}{2}$,
we define $\tau_1\in(\tau_0,\tau-1)$ by
 \[
 |z|=e^\frac{\tau-\tau_1}{2}.
 \]
We write $\Phi_2(z,\tau)$ as
 \[
 \Phi_2(\tau)
 =
 \int_{\tau_0}^{\tau_1}
 e^{A_z(\tau-\tau')}e^{-\tau'}G_\text{out}^\bot
 d\tau'
 +
 \int_{\tau_1}^{\tau-1}
 e^{A_z(\tau-\tau')}e^{-\tau'}G_\text{out}^\bot
 d\tau'
 +
 \int_{\tau-1}^\tau
 e^{A_z(\tau-\tau')}e^{-\tau'}G_\text{out}^\bot
 d\tau'.
 \]
Since $\tau_0<\tau_1<\tau-1$ and $|z|=e^\frac{\tau-\tau_1}{2}$,
we get from Lemma \ref{Lem3.6} and Lemma \ref{Lem7.3} that
 \begin{align*}
 \left|
 \int_{\tau_0}^{\tau_1}
 e^{A_z(\tau-\tau')}
 e^{-\tau'}G_\text{out}^\bot
 d\tau'
 \right|
 &\lesssim
 \int_{\tau_0}^{\tau_1}
 \left|
 e^{A_z(\tau-\tau_1)}e^{A_z(\tau-\tau'+\tau_1)}
 e^{-\tau'}G_\text{out}^\bot
 \right|
 d\tau'
 \\
 &\lesssim
 \int_{\tau_0}^{\tau_1}
 \frac{\frac{e^{-(\tau-\tau_1)}|z|^2}{4(1+e^{-\frac{1}{2}})}}
 {(1-e^{-(\tau-\tau_1)})^\frac{n}{4}}
 \left\|
 e^{A_z(\tau_1-\tau')}e^{-\tau'}G_\text{out}^\bot
 \right\|_\rho
 d\tau'
 \\
 &\lesssim
 \int_{\tau_0}^{\tau_1}
 e^{-(l+1)(\tau_1-\tau')}
 \left\|
 e^{-\tau'}G_\text{out}^\bot
 \right\|_\rho
 d\tau'
 \\
 &\lesssim
 \int_{\tau_0}^{\tau_1}
 e^{-(l+1)(\tau_1-\tau')}
 e^{-(2l+\frac{1}{2})\tau'}
 d\tau'
 \\
 &\lesssim
 e^{-(l+1)\tau}.
\end{align*}
We next estimate the second term.
We apply Lemma \ref{Lem7.5} to get
 \begin{align*}
 \int_{\tau_1}^\tau
 e^{A_z(\tau-\tau')}e^{-\tau'}|G_\text{\rm out}^\bot|d\tau'
 &\lesssim
 \frac{1}{R^\frac{1}{4}}
 \frac{e^{-l\tau_1}}{1+|y|^\frac{1}{4}}
 +
 T^be^{-l\tau_1}
 +
 e^{-(2l+1)\tau}
 |z|^{4l+4}
 +
 e^{-(l+\frac{2l+1}{2l+2})\tau}
 |z|^{2l+2}.
 \end{align*}
Since $|z|=e^\frac{\tau-\tau_1}{2}$,
it holds that $e^{-l\tau_1}=e^{-l\tau}|z|^{2l}$.
Therefore
we deduce that
 \begin{align*}
 \int_{\tau_1}^\tau
 e^{A_z(\tau-\tau')}e^{-\tau'}|G_\text{\rm out}^\bot|d\tau'
 &\lesssim
 \frac{1}{R^\frac{1}{4}}
 \frac{e^{-l\tau}|z|^{2l}}{1+|y|^\frac{1}{4}}
 +
 T^be^{-l\tau}|z|^{2l}
 +
 e^{-(2l+1)\tau}
 |z|^{4l+4}
 +
 e^{-(l+\frac{2l+1}{2l+2})\tau}
 |z|^{2l+2}
 \\
 &\lesssim
 \frac{e^{-l\tau}}{R^\frac{1}{4}}
 |z|^{2l}
 +
 e^{-(2l+1)\tau}
 |z|^{4l+4}
 +
 e^{-(l+\frac{2l+1}{2l+2})\tau}
 |z|^{2l+2}.
 \end{align*}
The estimate for the last term is derived from Lemma \ref{Lem7.5}.

\subsubsection{(iv) Estimate in $|z|>e^\frac{\tau-\tau_0}{2}$ and
$\tau\in(\tau_0+1,\infty)$}
\label{Sec7.4.4}
We get from Lemma \ref{Lem7.5} that
 \begin{align*}
 \int_{\tau_0}^\tau
 e^{A_z(\tau-\tau')}e^{-\tau'}|G_\text{\rm out}^\bot|d\tau'
 &<
 c
 \left(
 \frac{1}{R^\frac{1}{4}}
 \frac{e^{-l\tau_0}}{1+|y|^\frac{1}{4}}
 +
 T^be^{-l\tau_0}
 +
 e^{-(2l+1)\tau}|z|^{4l+4}
 \right.
 \\
 &\qquad
 \left.
 +
 e^{-(l+\frac{2l+1}{2l+2})\tau}
 |z|^{2l+2}
 \right)
 \qquad\text{\rm for}\ (z,\tau)\in\R^5\times(\tau_1,\infty).
 \end{align*}
Since $|z|>e^\frac{\tau-\tau_0}{2}$,
it holds that $e^{-l\tau_0}<e^{-l\tau}|z|^{2l}$.
Therefore
we obtain the desired estimate.
Combining estimates (i) - (iv),
we complete the proof of Proposition \ref{Pro7.1}.

\subsection{Pointwise estimate for $\Phi(z,\tau)$ in $|z|>e^\frac{l\tau}{2l+2}$}
\label{Sec7.5}
We go back to the function $\varphi(z,\tau)$ defined in \eqref{7.2}.
 \begin{lem}\label{Lem7.6}
 There exists $K_4>0$ independent of $R$, $\delta_0$, $\sigma$ such that
 \[
 |\varphi(z,\tau)|
 <
 K_4T^\frac{l}{l+1}
 \qquad\text{\rm for}\ |z|>e^\frac{l\tau}{2l+2},\ \tau>\tau_0.
 \]
 \end{lem}
\begin{proof}
We recall that
 \[
 \varphi(z,\tau)
 =
 {\bf b}(\tau)\cdot{\bf e}+\Phi_1(z,\tau)+\Phi_2(z,\tau).
 \]
We now check that
 \[
 \bar\varphi(z,\tau)
 =
 K\left( 2e^{-\frac{l\tau_0}{l+1}}-e^{-\frac{l\tau}{l+1}} \right)
 \]
gives a super solution in $|z|>e^\frac{l\tau}{2l+2}$.
From \eqref{7.4}, Lemma \ref{Lem7.2} and \eqref{7.16},
we note that
 \[
 |\varphi(z,\tau)|
 \lesssim
 e^{-(2l+1)\tau}
 |z|^{2l}
 +
 e^{-(l+1)\tau}
 |z|^{2l+2}
 +
 e^{-l\tau}|z|^{2l}
 +
 e^{-(l+\frac{2l+1}{2l+2})\tau}
 |z|^{2l+2}
 \]
for $|z|>1$.
Therefore
we get
 \[
 |\varphi(z,\tau)|
 \lesssim
 e^{-\frac{l\tau}{l+1}}
 \qquad \text{for } |z|=e^\frac{l\tau}{2l+2}.
 \]
Furthermore
by \eqref{7.5},
we see that the initial data satisfies
 \begin{align*}
 |\varphi(\tau_0)|
 &=
 |{\bf d}\cdot{\bf e}|
 \chi_\text{out}
 \lesssim
 e^{-(2l+1)\tau_0}
 \left( 1+|z|^{2l} \right)
 {\bf 1}_{|z|<2e^\frac{(l+\frac{1}{2})\tau_0}{2l+2}}
 \\
 &\lesssim
 e^{-(2l+1)\tau_0}
 \left( 1+|z|^{2l} \right)
 {\bf 1}_{|z|<e^\frac{\tau_0}{2}}
 \lesssim
 e^{-(l+1)\tau_0}.
 \end{align*}
Finally
we get from Lemma \ref{Lem7.1} that
 \begin{align*}
 e^{-\tau}|G_\text{out}|
 &\lesssim
 e^{-(l+\frac{2l+1}{2l+2})\tau}
 |z|^{2l+2}
 {\bf 1}_{e^\frac{(l+\frac{1}{2})\tau}{2l+2}<|z|<2e^\frac{(l+\frac{1}{2})\tau}{2l+2}}
 +
 \delta_0^2e^{-\tau}
 \\
 &\lesssim
 e^{-\frac{l\tau}{l+1}}
 \qquad\text{for } |z|>e^\frac{l\tau}{2l+2}.
\end{align*}
From this estimate,
we verify that if $K\gg1$
 \[
 \bar\varphi_\tau-A_z\bar\varphi
 =
 \frac{lK}{l+1}e^{-\frac{l\tau}{l+1}}
 \geq
 e^{-\tau}|G_\text{out}|
 \qquad\text{for } |z|>e^\frac{l\tau}{2l+2}.
 \]
Therefore
by a comparison argument,
we obtain
 \[
 |\varphi(z,\tau)|
 \lesssim
 \bar\varphi(z,\tau)
 \lesssim
 Ke^{-\frac{l\tau_0}{l+1}}
 =
 KT^\frac{l}{l+1}
 \qquad\text{for } |z|>e^\frac{l\tau}{2l+2},\ \tau>\tau_0.
 \]
The proof is completed.
\end{proof}
We now assume that
 \[
 T^\frac{l}{l+1}=e^{-\frac{lR}{l+1}}
 <
 \delta_0^2.
 \]
From this relation,
Lemma \ref{Lem7.6} is rewritten as
 \begin{align}\label{7.17}
 |\varphi(z,\tau)|
 <
 K_4\delta_0^2
 \qquad\text{for } |z|>e^\frac{l\tau}{2l+2},\ \tau>\tau_0.
 \end{align}
We finally derive estimates in $|x|>1$.
Let $W(x,t)$ be a solution of \eqref{7.1}.
 \begin{lem}\label{Lem7.7}
There exists $K_5>0$ independent of $R$, $\delta_0$, $\sigma$ such that
 \[
 |W(x,t)|
 <
 K_5\left( \frac{T^{3l+3}}{|x|^3}+\frac{\delta_0^2}{|x|^2} \right)
 \qquad\text{\rm for}\ |x|>1,\ t\in(0,T).
 \]
 \end{lem}
\begin{proof}
We now claim that
 \[
 \bar{W}(x,t)
 =
 \frac{K}{|x|^3}
 \left(
 T^{3l+3}-(T-t)^{3l+3}
 \right)
 +
 \frac{K\delta_0^2}{|x|^2}
 \]
gives a super solution in $|x|>1$.
Since $\varphi(z,\tau)=W(x,t)$ (see Section \ref{Sec7.1}),
from \eqref{7.17},
we verify that if $K>K_4$
 \begin{align*}
 |W(x,t)|
 <
 K_4\delta_0^2
 \leq
 \bar W(x,t)
 \qquad\text{for } |x|=1.
 \end{align*}Furthermore
from Lemma \ref{Lem7.1},
we recall that
 \[
 |G_\text{out}|
 \lesssim
 \frac{(T-t)^{3l+2}}{|x|^3}+\frac{\delta_0^2}{|x|^{2p}}
 \qquad\text{for } |x|>1.
 \]
Therefore
since $p>2$,
we get if $K\gg1$
 \[
 \bar W_t-\Delta_x\bar W
 =
 \frac{(3l+3)K}{|x|^3}(T-t)^{3l+2}
 +
 \frac{2K\delta_0^2}{|x|^4}
 \geq
 |G_\text{out}|
 \qquad\text{for } |x|>1.
 \]
Since $W(x,t)=0$ for $|x|>1$, $t=0$ (see \eqref{7.1}),
by a comparison argument,
we conclude
 \[
 |W(x,t)|
 <
 \bar W(x,t)
 \qquad\text{for } |x|>1,\ t\in(0,T).
 \]
This complete the proof.
\end{proof}

\section{Proof of Theorem \ref{Thm1}}
We perform the argument mentioned in Section \ref{Sec5.2} rigorously.
We fix $\delta_0\in(0,1)$ small enough.
Let $\tilde w(x,t)\in C(\R^5\times[0,T])$ be an extension of $w(x,t)\in X_\sigma$
defined in Section \ref{Sec5.2},
$(\lambda(t),\epsilon(y,t))$ be a pair of functions constructed
in Section \ref{Sec6},
and $W(x,t)$ be a function obtained in Section \ref{Sec7}.
We now define
 \[
 w(x,t)\in X_\sigma
 \mapsto
 W(x,t)\in C(\R^5\times[0,T))\cap C^{2,1}(\R^5\times(0,T)).
 \]
From Proposition \ref{Pro7.1} and Lemma \ref{Lem7.6} - Lemma \ref{Lem7.7},
there exists $R_0>0$ such that $W(x,t)\in X_\sigma$ if $R>R_0$.
By the way of construction,
the mapping
$w(x,t)\in(X_\sigma,d_{X_\sigma})\mapsto W(x,t)\in(X_\sigma,d_{X_\sigma})$
is continuous,
where $d_{X_\sigma}$ is the metric in $X_\sigma$ defined in \eqref{5.8}.
Furthermore
we note that
 \[
 \frac{\alpha_l^2}{2}(2\sigma)^{2l+2}\lambda(t)<2\alpha_2^2T^{2l+2}
 \qquad\text{and}\qquad
 \tau_0<\tau<|\log(2\sigma)|
 \qquad
 \text{for } t\in(0,T-2\sigma).
 \]
Therefore
from the local H\"older estimate and the decay of space infinity
$W(x,t)<\frac{\delta_0}{1+|x|^2}$ for $|x|>1$,
we find that the mapping is compact.
From the Schauder fixed point theorem,
we obtain a function $w(x,t)=W(x,t)\in X_\sigma$,
which is denoted by $w_\sigma(x,t)$.
Since $\tilde w_\sigma(x,t)=w_\sigma(x,t)$ for $t\in(0,T-2\sigma)$
(see Section \ref{Sec5.2}),
a pair of functions $(w_\sigma(x,t),\lambda_\sigma(t),\epsilon_\sigma(y,t))$
gives a solution of \eqref{5.5} for $t\in(0,T-2\sigma)$.
We take $\sigma\to0$ to obtain $w(x,t)=\lim_{\sigma\to0}w_\sigma(x,t)$.
This is the desired solution,
which has the form
 \[
 u(x,t)
 =
 {\sf Q}_{\lambda}
 +
 \Theta(x,t)\chi_\text{out}
 +
 \lambda^{-\frac{n-2}{2}}\epsilon(y,t)\chi_\text{in}
 +
 w(x,t)
 \]
with
 \[
 \left| \lambda(t)-\alpha_l^2(T-t)^{2l+2} \right|
 <
 K\left( \delta_0+\frac{T-t}{R^2} \right)
 (T-t)^{2l+2}.
 \]
If we repeat the argument for any $A<0$ (we fixed $A=-1$ in Section \ref{Sec4}),
we obtain the solution described in Theorem \ref{Thm1}.

\section*{Acknowledgement}
The author is partly supported by
Grant-in-Aid for Young Scientists (B) No. 26800065.


\end{document}